\documentclass{amsart}%
\usepackage{graphicx}
\usepackage{amscd}
\usepackage{amsmath}
\usepackage{amsfonts}
\usepackage{amssymb}%
\newtheorem{theorem}{Theorem}[section]

\newtheorem{corollary}[theorem]{Corollary}

\newtheorem{lemma}[theorem]{Lemma}

\newtheorem{proposition}[theorem]{Proposition}

\numberwithin{equation}{section}

\begin{document}
\title[Cobordisms of maps]{Cobordisms of maps with singularities of a\\given class}
\author{YOSHIFUMI ANDO}
\address{Department of Mathematical Sciences, Faculty of Science, Yamaguchi University,
Yamaguchi, 753-8512, Japan}
\email{andoy@yamaguchi-u.ac.jp}
\thanks{$2000$ \textit{Mathematics Subject Classification.} $57R45, 57R90, 58A20$}

\begin{abstract}
Let $P$ be a smooth manifold of dimension $p$. We will describe the group of
all cobordism classes of smooth maps of $n$-dimensional closed manifolds into
$P$ with singularities of a given class (including all fold singularities if
$n\geqq p$) in terms of certain stable homotopy groups by applying the
homotopy principle on the existence level, which is assumed to hold for those
smooth maps. We will also deal with the oriented version and construct a
classifying space of this oriented cobordism group in the dimensions $n<p$ and
$n\geqq p\geqq2$.

\end{abstract}
\maketitle


\section{Introduction}

Let $N$ and $P$ be smooth ($C^{\infty}$) manifolds of dimensions $n$ and $p$
respectively. Let $k\gg n,$ $p$ ($k$ may be $\infty$). Let $J^{k}(N,P)$ denote
the $k$-jet bundle of the manifolds $N$ and $P$ with the canonical projection
$\pi_{N}^{k}\times\pi_{P}^{k}$ onto $N\times P$ and the fiber $J^{k}(n,p)$ of
all$\ k$-jets of smooth map germs $(\mathbb{R}^{n},0)\rightarrow
(\mathbb{R}^{p},0)$. Here, $\pi_{N}^{k}$ and $\pi_{P}^{k}$ map a $k$-jet\ to
its source and target respectively. Let $\Omega=\Omega(n,p)$ denote a nonempty
open subspace of $J^{k}(n,p)$ which is invariant with respect to the action of
$L^{k}(p)$ $\times$ $L^{k}(n)$, where $L^{k}(m)$ denotes the group of $k$-jets
of germs of diffeomorphisms of $(\mathbb{R}^{m},0)$. Let $\Omega(N,P)$ denote
the open subbundle of $J^{k}(N,P)$ associated to $\Omega(n,p)$. A smooth map
$f:N\rightarrow P$ is called an $\Omega$-\emph{regular map} if and only if
$j^{k}f(N)\subset\Omega(N,P)$.

Let $C_{\Omega}^{\infty}(N,P)$\ denote the space of all $\Omega$-regular
maps\ of $N$ to $P$ equipped with the $C^{\infty}$ topology. Let
$\Gamma_{\Omega}(N,P)$ denote the space consisting of all continuous sections
of the fiber bundle $\pi_{N}^{k}|\Omega(N,P):\Omega(N,P)\rightarrow N$
equipped with the compact-open topology. Then we have the continuous map
$j_{\Omega}:C_{\Omega}^{\infty}(N,P)\rightarrow\Gamma_{\Omega}(N,P)$ defined
by $j_{\Omega}(f)=j^{k}f$. We say that $\Omega(n,p)$ satisfies \emph{the
homotopy principle} (simply \emph{h-principle}) if any section $s$ in
$\Gamma_{\Omega}(N,P)$ has an $\Omega$-regular map $f$ such that $j^{k}f$ is
homotopic to $s$ as sections. In this paper we say that $\Omega(n,p)$
satisfies \emph{the relative homotopy principle on the existence
level}\textit{ }if the following property (h-P) holds.

(h-P):\emph{ Let }$C$\emph{ be a closed subset of }$N$\emph{. Let }$s$\emph{
be a section of }$\Gamma_{\Omega}(N,P)$\emph{ which has an }$\Omega
$\emph{-regular map }$g$\emph{ defined on a neighborhood of }$C$\emph{ into
}$P$\emph{, where }$j^{k}g=s$\emph{. Then there exists an }$\Omega
$\emph{-regular map }$f:N\rightarrow P$\emph{ such that }$j^{k}f$\emph{ is
homotopic to }$s$\emph{ relative to a neighborhood of }$C$\emph{ by a homotopy
}$s_{\lambda}$\emph{ in }$\Gamma_{\Omega}(N,P)$\emph{ with }$s_{0}=s$\emph{
and }$s_{1}=j^{k}f$\emph{.}

Let $i_{+m}:J^{k}(n,p)\rightarrow J^{k}(n+m,p+m)$ denote the map defined by
$i_{+m}(j_{0}^{k}f)=j_{0}^{k}(f\times id_{\mathbb{R}^{m}})$ for a positive
integer, where $f:(\mathbb{R}^{n},0)\rightarrow(\mathbb{R}^{p},0)$ and
$id_{\mathbb{R}^{m}}$\ is the identity of $\mathbb{R}^{m}$. Let $\Omega
_{\star}=\Omega_{\star}(n+1,p+1)$ denote a nonempty open subspace of
$J^{k}(n+1,p+1)$ which is invariant with respect to the action of $L^{k}(p+1)$
$\times$ $L^{k}(n+1)$ and satisfies $i_{+1}(\Omega(n,p))\subset\Omega_{\star
}(n+1,p+1)$.\ Let $P$ be a smooth manifold of dimension $p$. We define the
notion of $\Omega_{\star}$-cobordisms of $\Omega$-regular maps to $P$. Let
$f_{i}:N_{i}\rightarrow P$ ($i=0,1$) be two $\Omega$-regular maps, where
$N_{i}$ are closed smooth $n$-dimensional manifolds. We say that they are
$\Omega_{\star}$\emph{-cobordant} when there exists an $\Omega_{\star}%
$-regular map, called an $\Omega_{\star}$-cobordism, $\mathfrak{C}:(W,\partial
W)\rightarrow(P\times\lbrack0,1],P\times0\cup P\times1)$ such that, for a
sufficiently small positive real number $\epsilon$,

(i) $W$ is a compact smooth manifold of dimension $n+1$ with $\partial W$
being $N_{0}\cup N_{1}$ and the collar of $\partial W$ is identified with
$N_{0}\times\lbrack0,\epsilon]\cup N_{1}\times\lbrack1-\epsilon,1]$,

(ii) $\mathfrak{C}|N_{0}\times\lbrack0,\epsilon]=f_{0}\times id_{[0,\epsilon
]}$ and $\mathfrak{C}|N_{1}\times\lbrack1-\epsilon,1]=f_{1}\times
id_{[1-\epsilon,1]}.$

\noindent We similarly define the notion of \emph{oriented} $\Omega_{\star}%
$-cobordisms of $\Omega$-regular maps by providing manifolds concerned with
orientations, where $N_{0}\cup N_{1}$ in (i) should be replaced by $N_{0}%
\cup(-N_{1})$. Let $\mathfrak{N}(n,P;\Omega,\Omega_{\star})$ (respectively
$\mathfrak{O}(n,P;\Omega,\Omega_{\star})$) denote the monoid of all
$\Omega_{\star}$-cobordism (respectively oriented $\Omega_{\star}$-cobordism)
classes of $\Omega$-regular maps to $P$. In this paper we will describe these
monoids of cobordism classes in terms of certain stable homotopy groups.

We need some notations for this purpose. Let $\mathcal{E}\rightarrow X$ and
$\mathcal{F}\rightarrow Y\,$be smooth vector bundles of dimensions $n$ and $p$
over smooth manifolds, and let $\pi_{X}$ and $\pi_{Y}$ be the projections of
$X\times Y$ onto $X$ and $Y$ respectively.\ Define the vector bundle
$J^{k}(\mathcal{E},\mathcal{F})$ over $X\times Y$ by%
\begin{equation}
J^{k}(\mathcal{E},\mathcal{F})=\bigoplus_{i=1}^{k}\text{\textrm{Hom}}%
(S^{i}(\pi_{X}^{\ast}(\mathcal{E})),\pi_{Y}^{\ast}(\mathcal{F}))
\end{equation}
with the canonical projections $\pi_{X}^{k}:J^{k}(\mathcal{E},\mathcal{F}%
)\rightarrow X$\ and $\pi_{Y}^{k}:J^{k}(\mathcal{E},\mathcal{F})\rightarrow
Y$. Here, $S^{i}(\mathcal{E})$ is the vector bundle $\cup_{x\in X}%
S^{i}(\mathcal{E}_{x})$ over $X$, where $S^{i}(\mathcal{E}_{x})$ denotes the
$i$-fold symmetric product of the fiber $\mathcal{E}_{x}$ over $x$. The
canonical fiber $\bigoplus_{i=1}^{k}\mathrm{Hom}(S^{i}(\mathbb{R}%
^{n}),\mathbb{R}^{p})$ is canonically identified with $J^{k}(n,p)$. If we
provide $N$ and $P$ with Riemannian metrics, then $J^{k}(TN,TP)$ is identified
with $J^{k}(N,P)$\ over $N\times P$ (see Section 2). Let $\Omega
(\mathcal{E},\mathcal{F})$ denote the open subbundle of $J^{k}(\mathcal{E}%
,\mathcal{F})$ associated to $\Omega(n,p)$.

Let $G_{m}$ refer to the grassmann manifold $G_{m,\ell}$\ (respectively
oriented grassmann manifold $\widetilde{G}_{m,\ell}$) of all $m$-subspaces
(respectively oriented $m$-subspaces) of $\mathbb{R}^{\ell+m}$.\ Let
$\gamma_{G_{m}}^{m}$ and $\widehat{\gamma}_{G_{m}}^{\ell}$\ denote the
canonical vector bundles of dimensions $m$ and $\ell$\ over the space $G_{m}%
$\ respectively such that $\gamma_{G_{m}}^{m}\oplus\widehat{\gamma}_{G_{m}%
}^{\ell}$ is the trivial bundle $\varepsilon_{G_{m}}^{\ell+m}$. Let
$T(\widehat{\gamma}_{G_{m}}^{\ell})$ denote the Thom space of $\widehat
{\gamma}_{G_{m}}^{\ell}$. The spaces $\{T(\widehat{\gamma}_{G_{m}}^{\ell
})\}_{\ell}$\ constitute a spectrum.\ Let $i^{G}:G_{n}\rightarrow G_{n+1}$
denote the injection mapping an $n$-plane $a$ to the $(n+1)$-plane generated
by $a$ and the $(n+\ell+1)$-th unit vector $\mathbf{e}_{n+\ell+1}$ in
$\mathbb{R}^{n+\ell+1}$. Let $\boldsymbol{\Omega}=\Omega(\gamma_{G_{n}}%
^{n},TP)$ and $\boldsymbol{\Omega}_{\ast}=\Omega_{\star}(\gamma_{G_{n+1}%
}^{n+1},TP\oplus\varepsilon_{P}^{1})$ be the open subbundles of $J^{k}%
(\gamma_{G_{n}}^{n},TP)$ and $J^{k}(\gamma_{G_{n+1}}^{n+1},TP\oplus
\varepsilon_{P}^{1})$ associated to $\Omega(n,p)$ and $\Omega_{\star
}(n+1,p+1)$ respectively, where $\varepsilon_{P}^{1}$ denotes the trivial
bundle $P\times\mathbb{R}$. Set $\widehat{\gamma}_{\boldsymbol{\Omega}}^{\ell
}=(\pi_{G_{n}}^{k})^{\ast}(\widehat{\gamma}_{G_{n}}^{\ell}%
)|_{\boldsymbol{\Omega}}$ and $\widehat{\gamma}_{\boldsymbol{\Omega}_{\ast}%
}^{\ell}=(\pi_{G_{n+1}}^{k})^{\ast}(\widehat{\gamma}_{G_{n+1}}^{\ell
})|_{\boldsymbol{\Omega}_{\ast}}$. There exists a fiberwise map $\Delta
^{(\Omega,\Omega_{\star})}:\boldsymbol{\Omega}\rightarrow\boldsymbol{\Omega
}_{\star}$ associated to $i_{+1}|\Omega(n,p):\Omega(n,p)\rightarrow
\Omega_{\star}(n+1,p+1)$ covering $i^{G}\times id_{P}$. Then $\Delta
^{(\Omega,\Omega_{\star})}$ induces the bundle map
\[
\mathbf{b}(\widehat{\gamma})^{(\Omega,\Omega_{\star})}:\widehat{\gamma
}_{\boldsymbol{\Omega}}^{\ell}\longrightarrow\widehat{\gamma}%
_{\boldsymbol{\Omega}_{\star}}^{\ell}%
\]
covering $\Delta^{(\Omega,\Omega_{\star})}$ and the associated map
$T(\mathbf{b}(\widehat{\gamma})^{(\Omega,\Omega_{\star})})$ between the Thom
spaces. Let $\ell\gg n,p$.\ We denote the image of%
\begin{equation}
T(\mathbf{b}(\widehat{\gamma})^{(\Omega,\Omega_{\star})})_{\ast}:\lim
_{\ell\rightarrow\infty}\pi_{n+\ell}\left(  T(\widehat{\gamma}%
_{\boldsymbol{\Omega}}^{\ell})\right)  \longrightarrow\lim_{\ell
\rightarrow\infty}\pi_{n+\ell}\left(  T(\widehat{\gamma}_{\boldsymbol{\Omega
}_{\star}}^{\ell})\right)
\end{equation}
by $\mathrm{\operatorname{Im}}^{\mathfrak{N}}\left(  T(\mathbf{b}%
(\widehat{\gamma})^{(\Omega,\Omega_{\star})})\right)  $ (respectively
$\mathrm{\operatorname{Im}}^{\mathfrak{O}}\left(  T(\mathbf{b}(\widehat
{\gamma})^{(\Omega,\Omega_{\star})})\right)  $) in the unoriented
(respectively\ the oriented) case.

We are ready to state the main result of this paper.

\begin{theorem}
Let $n$ and $p$ be positive integers. Let $P$ be a $p$-dimensional manifold.
Let $\Omega(n,p)$ and $\Omega_{\star}(n+1,p+1)$\ denote the above open subsets
invariant with respect to the actions of $L^{k}(p)$ $\times$ $L^{k}(n)$ and
$L^{k}(p+1)$ $\times$ $L^{k}(n+1)$, respectively, such that $i_{+1}%
(\Omega(n,p))\subset\Omega_{\star}(n+1,p+1)$, $\Omega(n,p)$ satisfies the
h-principle and that $\Omega_{\star}(n+1,p+1)$ satisfies the relative
h-principle on the existence level in (h-P).

Then there exist isomorphisms%
\begin{align*}
\mathfrak{n}^{(\Omega,\Omega_{\star})}  &  :\mathfrak{{N}}(n,P;\Omega
,\Omega_{\star})\longrightarrow\mathrm{\operatorname{Im}}%
\mathfrak{^{\mathfrak{N}}}\left(  T(\mathbf{b}(\widehat{\gamma})^{(\Omega
,\Omega_{\star})})\right)  ,\\
\mathfrak{o}^{(\Omega,\Omega_{\star})}  &  :\mathfrak{{O}}(n,P;\Omega
,\Omega_{\star})\longrightarrow\mathrm{\operatorname{Im}}%
\mathfrak{^{\mathfrak{O}}}\left(  T(\mathbf{b}(\widehat{\gamma})^{(\Omega
,\Omega_{\star})})\right)  .
\end{align*}

\end{theorem}

Let $\mathcal{K}$ denote the contact group introduced in \cite{MathIII} which
acts on $J^{k}(n,p)$ or $J^{k}(n+m,p+m)$. Let $\Omega_{\mathcal{K}}^{m}%
=\Omega_{\mathcal{K}}(n+m,p+m)$ denote the subset of $J^{k}(n+m,p+m)$ which
consists of all $\mathcal{K}$-orbits $\mathcal{K}(i_{+m}(z))$ for $k$-jets
$z\in\Omega(n,p)$. It will be proved that $\Omega_{\mathcal{K}}(n+m,p+m)$ is
an open subset (see Lemma 5.2).

As for the image of $T(\mathbf{b}(\widehat{\gamma})^{(\Omega,\Omega_{\star}%
)})_{\ast}$, we will prove the following theorem.

\begin{theorem}
Let $n<p$ and $P$ be as in Theorem 1.1. Let $\Omega(n,p)$ denote a nonempty
open subset in $J^{k}(n,p)$ invariant with respect to the action of
$\mathcal{K}$ and let $\Omega_{\mathcal{K}}(n+1,p+1)$ be as above. Then the
homomorphism $T(\mathbf{b}(\widehat{\gamma})^{(\Omega,\Omega_{\mathcal{K}}%
)})_{\ast}$ in (1.2) is surjective.
\end{theorem}

Let $\Omega_{\mathcal{K}}(\gamma_{\widetilde{G}_{n+m,\ell}}^{n+m}%
,\mathbb{R}^{p+m})$ denote the subbundle of $J^{k}(\gamma_{\widetilde
{G}_{n+m,\ell}}^{n+m},\mathbb{R}^{p+m})$\ associated to $\Omega_{\mathcal{K}%
}(n+m,p+m)$.\ For $\mathfrak{{O}}=\mathfrak{{O}}(n,P;\Omega,\Omega
_{\mathcal{K}})$ we define%
\[
B_{\mathfrak{O}}=\lim_{m\rightarrow\infty}\left(  \lim_{\ell\rightarrow\infty
}C^{0}(S^{\ell+n-p},T((\pi_{\widetilde{G}_{n+m,\ell}}^{k})^{\ast}%
\widehat{\gamma}_{\widetilde{G}_{n+m,\ell}}^{\ell}|_{\Omega_{\mathcal{K}%
}(\gamma_{\widetilde{G}_{n+m,\ell}}^{n+m},\mathbb{R}^{p+m})}))\right)  ,
\]
where $C^{0}(X,Y)$ denotes the space consisting of\ all base-point preserving
continuous maps between connected spaces with base points equipped with the
compact-open topology.

\begin{theorem}
Let $n<p$ or $n\geqq p\geqq2$. Let $P$ be a closed connected oriented
$p$-dimensional manifold. Let $\Omega(n,p)$ be a nonempty $\mathcal{K}%
$-invariant open subset such that if $n\geqq p\geqq2$, then $\Omega(n,p)$
contains all fold jets at least. Then there exists an isomorphism%
\[
\mathfrak{{O}}(n,P;\Omega,\Omega_{\mathcal{K}})\longrightarrow\lbrack
P,B_{\mathfrak{O}}].
\]

\end{theorem}

In our study of cobordisms of singular maps the h-principle for $\Omega
$-regular maps plays a quite important role and our method is available in the
dimensions $n\geqq p$ as well. Theorem 1.1 has been developed from an
observation for fold-maps in \cite[Theorem 0.4]{FoldRIMS}. Set
$\boldsymbol{\Omega}_{\mathcal{K}}^{m}=\Omega_{\mathcal{K}}(\gamma_{G_{n+m}%
}^{n+m},TP\oplus\varepsilon_{P}^{m})$. Theorem 1.2 implies that $\mathfrak{{O}%
}(n,P;\Omega,\Omega_{\mathcal{K}})$\ is isomorphic to the stable homotopy
group $\lim_{\ell\rightarrow\infty}\pi_{n+\ell}\left(  T(\widehat{\gamma
}_{\boldsymbol{\Omega}_{\mathcal{K}}^{1}}^{\ell})\right)  $. If we apply the
relative h-principle in Theorem 9.1 to \cite[Theorem 9.2]{Sady2} due to
Sadykov, then it follows that $\mathfrak{{O}}(n,P;\Omega,\Omega_{\mathcal{K}%
})$\ is isomorphic to $\pi_{n+\ell}\left(  T(\widehat{\gamma}%
_{\boldsymbol{\Omega}_{\mathcal{K}}^{m}}^{\ell})\right)  $ also in the
dimensions $n\geqq p\geqq2$ as well as $n<p$, where $\widehat{\gamma
}_{\boldsymbol{\Omega}_{\mathcal{K}}^{m}}^{\ell}=(\pi_{G_{n+m}}^{k})^{\ast
}(\widehat{\gamma}_{G_{n+m}}^{\ell})|_{\boldsymbol{\Omega}_{\mathcal{K}}^{m}}$
and $\ell$ and $m$ are sufficiently large integers. By using these stable
homotopy groups, we will induce the space $B_{\mathfrak{O}}$ by applying the
$S$-duality in \cite{SpaDual} in Section 8. In Section 9 we will see that the
h-principle in (h-P) holds for a very wide class of $\mathcal{K}$-invariant
open sets $\Omega(n,p)$. Let $\Sigma^{I}(n,p)$ denote the Boardman manifold in
$J^{k}(n,p)$ with Boardman symbol $I=(i_{1},\cdots,i_{k})$ defined in
\cite{Board}. An important example of $\Omega(n,p)$ is the open subset
$\Omega^{I}(n,p)$ which is the union of all $\Sigma^{K}(n,p)$ with $K\leqq I$
in the lexicographic order. Th results of the paper show the importance of the
homotopy type of $\Omega(n,p)$. In \cite{Omega10}, \cite{FoldSurg},
\cite{FoldRIMS} and \cite{Bo10Topo} we have described the homotopy type of
$\Omega^{i,0}(n,p)$\ for $i=\max\{n-p+1,1\}$ in terms of orthogonal groups and
Stiefel manifolds. Although available only in the case of fold-maps, we will
construct a simpler spectrum in Section 10 whose stable homotopy group is a
direct summand of $\mathfrak{O}(n,P;\Omega^{i,0},\Omega^{i,0})$ (respectively
$\mathfrak{N}(n,P;\Omega^{i,0},\Omega^{i,0})$) and also construct the
corresponding classifying space. This result is a refinement of \cite[Theorem
0.3]{FoldRIMS}. Let $F_{m}$ denote the space of all base point preserving maps
of the $m$-sphere $S^{m}$ and let $F=\lim_{m\rightarrow\infty}F_{m}$. When
$n=p\geqq1$ and $P$ is closed and oriented, we will prove that there exists an
isomorphism of $\mathfrak{O}(n,P;\Omega^{1,0},\Omega^{1,0})$ onto $[P,F]$,
which has been given in \cite{FoldKyoto} and \cite{FoldRIMS} from a different
point of view. Therefore, we may assert that the topology of $B_{\mathfrak{O}%
}$ will be important in connection with the canonical homomorphism of
$\mathfrak{O}(n,S^{n};\Omega^{1,0},\Omega^{1,0})$ to $\mathfrak{O}%
(n,S^{n};\Omega,\Omega_{\mathcal{K}})$ and the map $F\rightarrow
B_{\mathfrak{O}}$. We should refer to Chess\cite{Che}, in which $\mathfrak{O}%
(n,\mathbb{R}^{n};\Omega^{1,0},\Omega^{1,0})\approx\pi_{n}^{s}$ has been proved

The study of h-principles for $\Omega$-regular maps has a long history. Here,
we only refer the reader to Smale\cite{Smale}, Hirsch\cite{HirsImm},
Phillips\cite{P},\ Feit\cite{Feit}, Gromov(\cite{G1}, \cite{G2}),
\`{E}lia\v{s}berg(\cite{E1}, \cite{E2}) and du Plessis(\cite{duPMCS},
\cite{duPRS}, \cite{duPContIn}) for the details and further references.

In \cite{Eliash} Eliashberg has studied the cobordisms of the solutions of the
first order differential relations such as Lagrange and Legendre immersions by
applying his h-principle. Sadykov\cite{Sady2} has studied and expressed the
cobordism group of $\Omega$-regular maps in terms of the stable homotopy group
as explained above under the assumption of the relative h-principle in a
formal approach.

In \cite{R-S} Rim\'{a}nyi and Sz\H{u}cs have constructed a certain classifying
space such that the group of the cobordism classes of smooth maps of
$n$-dimensional manifolds into $P$ having only a given class of $C^{\infty}$
stable singularities is described by the homotopy classes of $P$ to this space
in the case $n<p$ and in \cite{Szu} Sz\H{u}cs has developed further results.
The method of the construction of this classifying space is quite different
from ours of the space $B_{\mathfrak{O}}$ using h-principles in this paper. In
\cite{KalmarCob}, \cite{KalmarNEW} and \cite{KalmarNEW2}\ Kalm\'{a}r has
studied the cobordism groups of fold-maps in negative codimensions.

As another line of investigation of cobordisms of singular maps in which
h-principles are not available, we refer to Saeki\cite{Saek1},
Ikegami-Saeki\cite{I-S}, Saeki\cite{Saek}, Kalm\'{a}r\cite{Kalmar} and
Sadykov\cite{SadyP}.

We will define the homomorphisms $\mathfrak{n}^{(\Omega,\Omega_{\star})}$\ and
$\mathfrak{o}^{(\Omega,\Omega_{\star})}$ in Section 3. In Section 4 we will
prove Theorem 1.1. In Section 5 we will prove that $\Omega_{\mathcal{K}}^{m}$
is an open subset. In Section 6 we will prepare several results which are
necessary in the proofs of Theorems 1.2 and 1.3. In Section 7 we will prove
Theorem 1.2. In Section 8 we will explain how the space $B_{\mathfrak{O}}$ is
introduced. In Section 9 we will give a wide class of open sets $\Omega(n,p)$
which satisfy the h-principle in (h-P). In Section 10 we will construct a
simpler spectrum and then, the classifying space $B_{V}$ such that $[P,B_{V}]$
is a direct summand of the cobordism group for fold-maps by using Theorem 1.1
and the homotopy type of $\Omega^{n-p+1,0}(n,p)$.

\section{Preliminaries}

Given a fiber bundle $\pi^{\mathcal{E}}:\mathcal{E}\rightarrow X$ and a subset
$C$ in $X,$ we denote $(\pi^{\mathcal{E}})^{-1}(C)$ by $\mathcal{E}|_{C}.$ Let
$\pi^{\mathcal{F}}:\mathcal{F}\rightarrow Y$ be another fiber bundle. A map
$\tilde{b}:\mathcal{E}\rightarrow\mathcal{F}$ is called a fiber map covering a
map $b:X\rightarrow Y$ if $\pi^{\mathcal{F}}\circ\tilde{b}=b\circ
\pi^{\mathcal{E}}$ holds. The restriction $\tilde{b}|(\mathcal{E}%
|_{C}):\mathcal{E}|_{C}\rightarrow\mathcal{F}$ (or $\mathcal{F}|_{b(C)}$) is
denoted by $\tilde{b}|_{C}$. In particular, for a point $x\in X,$
$\mathcal{E}|_{x}$ and $\tilde{b}|_{x}$ are simply denoted by $\mathcal{E}%
_{x}$ and $\tilde{b}_{x}:\mathcal{E}_{x}\rightarrow\mathcal{F}_{b(x)}$
respectively. The trivial bundle $X\times\mathbb{R}^{\ell}$ is denoted by
$\varepsilon_{X}^{\ell}$ (\cite{Ste}).

Let $\mathcal{E}\rightarrow X$ and $\mathcal{F}\rightarrow Y$ be vector
bundles of dimensions $n$ and $p$ respectively. The origin of $\mathbb{R}^{m}$
is\ simply denoted by $0$ for any natural number $m$. We define the action of
$L^{k}(p)$ $\times$ $L^{k}(n)$ on $J^{k}(n,p)$ by $(j_{0}^{k}h_{1},j_{0}%
^{k}h_{2})\cdot j_{0}^{k}f=j_{0}^{k}(h_{1}\circ f\circ h_{2}^{-1})$. Let
$\mathcal{E}_{1}\rightarrow X_{1}$ and $\mathcal{F}_{1}\rightarrow Y_{1}$ be
other vector bundles of dimensions $n$ and $p$ respectively, and let
$\tilde{b}_{1}:\mathcal{E}\rightarrow\mathcal{E}_{1}$ and $\tilde{b}%
_{2}:\mathcal{F}\rightarrow\mathcal{F}_{1}$ be bundle maps covering
$b_{1}:X\rightarrow X_{1}$ and $b_{2}:Y\rightarrow Y_{1}$ respectively. Then
$\tilde{b}_{1}$ and $\tilde{b}_{2}$\ yield the isomorphisms $S^{i}%
(\mathcal{E}_{x})\rightarrow S^{i}(\mathcal{E}_{1,b_{1}(x)})$ and
$\mathcal{F}_{y}\rightarrow\mathcal{F}_{1,b_{2}(y)}$\ for any $x\in X$ and
$y\in Y$ for $1\leqq i\leqq k$ and hence, we have the bundle map
\begin{equation}
\mathbf{j}(\tilde{b}_{1},\tilde{b}_{2}):J^{k}(\mathcal{E},\mathcal{F}%
)\rightarrow J^{k}(\mathcal{E}_{1},\mathcal{F}_{1})
\end{equation}
covering $b_{1}\times b_{2}$. Then $\mathbf{j}(\tilde{b}_{1},\tilde{b}_{2})$
induces the bundle map $\mathbf{j}(\tilde{b}_{1},\tilde{b}_{2})_{\Omega
}:\Omega(\mathcal{E},\mathcal{F})\rightarrow\Omega(\mathcal{E}_{1}%
,\mathcal{F}_{1})$.

If we provide $N$ and $P$ with Riemannian metrics, then the Levi-Civita
connections induce the exponential maps $\exp_{N,x}:T_{x}N\rightarrow N$ and
$\exp_{P,y}:T_{y}P\rightarrow P$ for $x\in N$ and $y\in P$ respectively. In
dealing with the exponential maps we always consider the convex neighborhoods
(\cite{KoNo}). We define the smooth bundle map
\begin{equation}
J^{k}(N,P)\mathbf{\rightarrow}J^{k}(TN,TP)\text{ \ \ \ over }N\times P
\end{equation}
by sending $z=j_{x}^{k}f\in(\pi_{N}^{k}\times\pi_{P}^{k})^{-1}(x,y)$ to the
$k$-jet of $(\exp_{P,y})^{-1}\circ f\circ\exp_{N,x}$ at $\mathbf{0}\in T_{x}%
N$, which is regarded as an element of $J^{k}(T_{x}N,T_{y}P)(=J_{x,y}%
^{k}(TN,TP))$ (see \cite[Proposition 8.1]{KoNo} for the smoothness of
exponential maps). More strictly, (2.2) gives a smooth equivalence of the
fiber bundles under the structure group $L^{k}(p)\times L^{k}(n)$. Namely, it
gives a smooth reduction of the structure group $L^{k}(p)\times L^{k}(n)$ of
$J^{k}(N,P)$ to $O(p)\times O(n)$, which is the structure group of
$J^{k}(TN,TP)$. Let $\Sigma^{I}(N,P)$ denote the Boardman manifold in
$J^{k}(N,P)$, where $I=(i_{1},\cdots,i_{k})$ is a Boardman symbol such that
$n\geqq i_{1}\geqq\cdots\geqq i_{k}\geqq0$ (see \cite{Board}, \cite{LevSDM}
and \cite{MathTB}). Let $\Omega^{I}(N,P)$ denote the open subset which is the
union of all $\Sigma^{K}(N,P)$ with $K\leqq I$ in the lexicographic order. Let
$\Sigma^{I}(TN,TP)$ and $\Omega^{I}(TN,TP)$ be the subbundles of
$J^{k}(TN,TP)$ associated to $\Sigma^{I}(n,p)$ and $\Omega^{I}(n,p)$, which
are identified with $\Sigma^{I}(N,P)$ and $\Omega^{I}(N,P)$ under (2.2), respectively.

Let $f:X\rightarrow Y$ be a continuous map. If $f_{\ast}:\pi_{i}%
(X)\rightarrow\pi_{i}(Y)$ is an isomorphism for $0\leqq i<m$ and an
epimorphism for $i=m$, then we call $f$ a homotopy $m$-equivalence in this paper.

\section{Homomorphisms $\mathfrak{n}^{(\Omega,\Omega_{\star})}$ and
$\mathfrak{o}^{(\Omega,\Omega_{\star})}$}

Let $\Omega(n,p)$ and $\Omega_{\star}(n+1,p+1)$ be open subsets given in
Theorem 1.1.

As usual we provide $\mathfrak{N}(n,P;\Omega,\Omega_{\star})$ and
$\mathfrak{O}(n,P;\Omega,\Omega_{\star})$\ with the structures of monoids.
Namely, given two $\Omega$-regular maps $f_{i}:N_{i}\rightarrow P$ ($i=0,1$),
we define the sum $[f_{0}]+[f_{1}]$ to be the cobordism class of the $\Omega
$-regular map $f:N_{0}\cup N_{1}\rightarrow P$ defined by $f|N_{i}=f_{i}$. We
proceed the arguments commonly in the unoriented case and the oriented case in
Sections 3 and 4, because in the oriented case we only need to provide
manifolds and vector bundles concerned with the orientations. The
orientability of $P$ is not necessary even in the oriented case in Theorems
1.1 and 1.2. We will need the orientability of $P$ at the end of Section 8.

We will identify $\mathbb{R}^{m}$\ with $S^{m}\backslash\{(0,\cdots,0,1)\}$
for positive integers $m$ in the following. Let $M$ be an $m$-dimensional
compact manifold such that $M$ should be oriented in the oriented case. Let
$\ell\gg m$. Take an embedding $e_{M}:M\rightarrow\mathbb{R}^{\ell+m}$ and
identify $M$ with $e_{M}(M)$. Let $c_{M}:M\rightarrow G_{m}$ be the
classifying map defined by sending a point $x\in M$ to the $m$-plane
$T_{x}M\in G_{m}$. Let $\nu_{M}$\ be the orthogonal normal bundle of $M$ in
$\mathbb{R}^{\ell+m}$. Let $c_{TM}:TM\rightarrow\gamma_{G_{m}}^{m}$
(respectively $c_{\nu_{M}}:\nu_{M}\rightarrow\widehat{\gamma}_{G_{m}}^{\ell}$)
be the bundle map covering the classifying map $c_{M}:M\rightarrow G_{m}$,
which is defined by sending a vector $\mathbf{v}$ of $T_{x}M$ (respectively
$\mathbf{w\in}\nu_{M,x}$) to $(T_{x}M,\mathbf{v})$ (respectively
$(T_{x}M,\mathbf{w})$). Then we have the canonical trivializations
$t_{M}:TM\oplus\nu_{M}\rightarrow\varepsilon_{M}^{\ell+m}$ and $t_{G_{m}%
}:\gamma_{G_{m}}^{m}\oplus\widehat{\gamma}_{G_{m}}^{\ell}\rightarrow
\varepsilon_{G_{m}}^{\ell+m}$. They induce $t_{G_{m}}\circ(c_{TM}\oplus
c_{\nu_{M}})\circ t_{M}^{-1}=c_{M}\times id_{\mathbb{R}^{\ell+m}}$. Let
$\mathcal{F}$ be any vector bundle of dimension $q$\ over $P$ and
$\Omega(m,q)$ be an\ $L^{k}(q)$ $\times$ $L^{k}(m)$-invariant open subspace of
$J^{k}(m,q)$. Let $\boldsymbol{\Omega}_{m}=\Omega(\gamma_{G_{m}}%
^{m},\mathcal{F})$ be the subbundle of $J^{k}(\gamma_{G_{m}}^{m},\mathcal{F})$
associated to $\Omega(m,q)$. If there is a map $s_{M}:M\rightarrow
\boldsymbol{\Omega}_{m}$ with $\pi_{G_{m}}^{k}\circ s_{M}$ being homotopic
(respectively equal) to $c_{M}$, then $c_{TM}$, $c_{\nu_{M}}$ and the
projection $\pi_{G_{m}}^{k}|\boldsymbol{\Omega}_{m}:\boldsymbol{\Omega}%
_{m}\rightarrow G_{m}$ induce the bundle maps $c_{s_{M}}^{TM}:TM\rightarrow
(\pi_{G_{m}}^{k})^{\ast}\gamma_{G_{m}}^{m}|_{\boldsymbol{\Omega}_{m}}$ and
$c_{s_{M}}^{\nu_{M}}:\nu_{M}\rightarrow(\pi_{G_{m}}^{k})^{\ast}\widehat
{\gamma}_{G_{m}}^{\ell}|_{\boldsymbol{\Omega}_{m}}$ covering $s_{M}$ such that
$t_{\boldsymbol{\Omega}_{m}}\circ(c_{s_{M}}^{TM}\oplus c_{s_{M}}^{\nu_{M}%
})\circ t_{M}^{-1}$ is homotopic (respectively equal) to $s_{M}\times
id_{\mathbb{R}^{\ell+m}}$, where $t_{\mathbf{\Omega}_{m}}:(\pi_{G_{m}}%
^{k})^{\ast}(\gamma_{G_{m}}^{m}\oplus\widehat{\gamma}_{G_{m}}^{\ell
})|_{\boldsymbol{\Omega}_{m}}\rightarrow\varepsilon_{\boldsymbol{\Omega}_{m}%
}^{\ell+m}$ is the trivialization induced from $t_{G_{m}}$.%
\[%
\begin{array}
[c]{ccccc}
&  & \Omega(\gamma_{G_{m}}^{m},\mathcal{F}) & \longleftarrow & (\pi_{G_{m}%
}^{k})^{\ast}\gamma_{G_{m}}^{m}|_{\boldsymbol{\Omega}_{m}}\text{ (or }%
(\pi_{G_{m}}^{k})^{\ast}\widehat{\gamma}_{G_{m}}^{\ell}|_{\boldsymbol{\Omega
}_{m}}\text{)}\\
& \nearrow{\scriptsize s}_{{\scriptsize M}} & \downarrow &  & \text{
\ \ }\uparrow{\scriptsize c}_{{\scriptsize s}_{M}}^{{\scriptsize TM}%
}{\scriptsize (}\text{{\scriptsize or} }{\scriptsize c}_{{\scriptsize s}_{M}%
}^{{\scriptsize \nu}_{M}}{\scriptsize )}\\
M & \longrightarrow & G_{m}\times P &  & TM\text{ (or }\nu_{M}\text{)}\\
& \searrow{\scriptsize c}_{{\scriptsize M}} & \downarrow &  & \text{
\ \ }\downarrow{\scriptsize c}_{{\scriptsize TM}}{\scriptsize (}%
\text{{\scriptsize or} }{\scriptsize c}_{{\scriptsize \nu}_{M}}{\scriptsize )}%
\\
&  & G_{m} & \longleftarrow & \gamma_{G_{m}}^{m}\text{ (or }\widehat{\gamma
}_{G_{m}}^{\ell}\text{)}%
\end{array}
\]

Let $A_{m,q}$ refer to the grassmann manifold $G_{m,q}$\ (respectively
oriented grassmann manifold $\widetilde{G}_{m,q}$) of all $m$-subspaces
(respectively oriented $m$-subspaces) of $\mathbb{R}^{m+q}$.\ Let
$\gamma_{A_{m,q}}^{m}$ and $\widehat{\gamma}_{A_{m,q}}^{q}$\ denote the
canonical vector bundles of dimensions $m$ and $q$\ over $A_{m,q}%
$\ respectively such that $\gamma_{A_{m,q}}^{m}\oplus\widehat{\gamma}%
_{A_{m,q}}^{q}$ is the trivial bundle $\varepsilon_{A_{m,q}}^{m+q}$. Let
$T(\widehat{\gamma}_{A_{m,q}}^{q})$ denote the Thom space of $\widehat{\gamma
}_{A_{m,q}}^{q}$. Here, we see that the spaces $\{T((\pi_{G_{m}}^{k})^{\ast
}\widehat{\gamma}_{G_{m}}^{\ell}|_{\boldsymbol{\Omega}_{m}})\}_{\ell}$
constitute a spectrum. Let $j:G_{m}\rightarrow A_{m,\ell+1}$ denote the
injection mapping an $m$-plane $a$ in $\mathbb{R}^{m+\ell}=\mathbb{R}^{m+\ell
}\times0$\ to the same $a$ as the $m$-plane in $\mathbb{R}^{m+\ell+1}$. We
have the canonical bundle maps%
\[
b:\gamma_{G_{m}}^{m}\longrightarrow\gamma_{A_{m,\ell+1}}^{m}\text{ \ and
\ }\widehat{b}:\widehat{\gamma}_{G_{m}}^{\ell}\oplus\varepsilon_{G_{m}}%
^{1}\longrightarrow\widehat{\gamma}_{A_{m,\ell+1}}^{\ell+1}%
\]
covering $j$. Let $\boldsymbol{\Omega}_{m}^{\bullet}=\Omega(\gamma
_{A_{m,\ell+1}}^{m},\mathcal{F})$ be the open subbundle of $J^{k}%
(\gamma_{A_{m,\ell+1}}^{m},\mathcal{F})$ associated to $\Omega(m,q)$. We have
the bundle map%
\[
\mathbf{j}(b,id_{\mathcal{F}})_{\Omega}:\boldsymbol{\Omega}_{m}\longrightarrow
\boldsymbol{\Omega}_{m}^{\bullet}%
\]
covering $j$. Then $\widehat{b}$ induces the bundle map
\[
\mathbf{b}:(\pi_{G_{m}}^{k})^{\ast}(\widehat{\gamma}_{G_{m}}^{\ell}%
\oplus\varepsilon_{G_{m}}^{1})|_{\boldsymbol{\Omega}_{m}}\longrightarrow
(\pi_{A_{m,\ell+1}}^{k})^{\ast}(\widehat{\gamma}_{A_{m,\ell+1}}^{\ell
+1})|_{\boldsymbol{\Omega}_{m}^{\bullet}}%
\]
covering $\mathbf{j}(b,id_{\mathcal{F}})_{\Omega}$. Since $T((\pi_{G_{m}}%
^{k})^{\ast}(\widehat{\gamma}_{G_{m}}^{\ell}\oplus\varepsilon_{G_{m}}%
^{1})|_{\boldsymbol{\Omega}_{m}})=T((\pi_{G_{m}}^{k})^{\ast}(\widehat{\gamma
}_{G_{m}}^{\ell})|_{\boldsymbol{\Omega}_{m}})\wedge S^{1}$, we have the
associated map%
\[
T(\mathbf{b}):T((\pi_{G_{m}}^{k})^{\ast}(\widehat{\gamma}_{G_{m}}^{\ell
})|_{\boldsymbol{\Omega}_{m}})\wedge S^{1}\longrightarrow T((\pi_{A_{m,\ell
+1}}^{k})^{\ast}(\widehat{\gamma}_{A_{m,\ell+1}}^{\ell+1}%
)|_{\boldsymbol{\Omega}_{m}^{\bullet}}).
\]
This shows the assertion.

Take an embedding $e_{N}:N\rightarrow\mathbb{R}^{n+\ell}\mathbb{\subset
}S^{n+\ell}$ and apply the above notation. Then we have the bundle map%
\[
\mathbf{j}(c_{TN},id_{TP})_{\Omega}:\Omega(TN,TP)\rightarrow\boldsymbol{\Omega
}=\Omega(\gamma_{G_{n}}^{n},TP).
\]
Let $f:N\rightarrow P$ be an $\Omega$-regular map with the jet extension
$j^{k}f:N\rightarrow\Omega(TN,TP)$ and let $s_{N}=\mathbf{j}(c_{TN}%
,id_{TP})_{\Omega}\circ j^{k}f$. Then we have the composite%
\[
\Delta^{(\Omega,\Omega_{\star})}\circ s_{N}:N\rightarrow\boldsymbol{\Omega
}_{\star}=\Omega_{\star}(\gamma_{G_{n+1}}^{n+1},TP\oplus\varepsilon_{P}^{1})
\]
and%
\[
\mathbf{b}(\widehat{\gamma})^{(\Omega,\Omega_{\star})}\circ c_{s_{N}}^{\nu
_{N}}:\nu_{N}\rightarrow\widehat{\gamma}_{\boldsymbol{\Omega}_{\star}}^{\ell}%
\]
covering $\Delta^{(\Omega,\Omega_{\star})}\circ s_{N}$. We denote the
Pontrjagin-Thom construction for $e_{N}$ by $a_{N}:S^{n+\ell}\rightarrow
T(\nu_{N})$ (\cite{Thom}). Let $\mathfrak{C}(n,P;\Omega,\Omega_{\star}%
)$\ refer to $\mathfrak{N}(n,P;\Omega,\Omega_{\star})$\ or $\mathfrak{O}%
(n,P;\Omega,\Omega_{\star})$, let $\mathrm{\operatorname{Im}}\left(
T(\mathbf{b}(\widehat{\gamma})^{(\Omega,\Omega_{\star})})\right)  $\ refer to
$\mathrm{\operatorname{Im}}^{\mathfrak{N}}\left(  T(\mathbf{b}(\widehat
{\gamma})^{(\Omega,\Omega_{\star})})\right)  $\ or $\mathrm{\operatorname{Im}%
}^{\mathfrak{O}}\left(  T(\mathbf{b}(\widehat{\gamma})^{(\Omega,\Omega_{\star
})})\right)  $ and let $\omega$ refer to $\mathfrak{n}^{(\Omega,\Omega_{\star
})}$ or $\mathfrak{o}^{(\Omega,\Omega_{\star})}$, depending on whether we work
in the unoriented case or oriented case.\ We now define the maps%
\begin{equation}
\omega:\mathfrak{C}(n,P;\Omega,\Omega_{\star})\longrightarrow
\mathrm{\operatorname{Im}}\left(  T(\mathbf{b}(\widehat{\gamma})^{(\Omega
,\Omega_{\star})})\right) \nonumber
\end{equation}
by mapping the cobordism class $[f]$ of $\mathfrak{C}(n,P;\Omega,\Omega
_{\star})$ to the homotopy class of $T(\mathbf{b}(\widehat{\gamma}%
)^{(\Omega,\Omega_{\star})})\circ T(c_{s_{N}}^{\nu_{N}})\circ a_{N}$.

We have to prove that $\omega([f])$ does not depend on the choice of a
representative $f$.

\begin{lemma}
Suppose that two $\Omega$-regular maps $f_{i}:N_{i}\rightarrow P$ $(i=0,1)$
are $\Omega_{\star}$-cobordant. Then we have $\mathfrak{n}^{(\Omega
,\Omega_{\star})}([f_{0}])=\mathfrak{n}^{(\Omega,\Omega_{\star})}([f_{1}])$.
If $N_{i}$ are oriented and $f_{i}$ $(i=0,1)$ are oriented $\Omega_{\star}%
$-cobordant, then we have $\mathfrak{o}^{(\Omega,\Omega_{\star})}%
([f_{0}])=\mathfrak{o}^{(\Omega,\Omega_{\star})}([f_{1}])$.
\end{lemma}

\begin{proof}
We first have to prove that $\omega([f])$ does not depend on the choice of an
embedding $e_{N}$. Let $\ell$ be a sufficiently large integer. Let
$e_{N}:N\rightarrow\mathbb{R}^{n+\ell}$\ be an embedding of $N$. For a
nonnegative integer $m$, let $e_{N}^{m}$ be the composite of $e_{N}$ and the
inclusion $\mathbb{R}^{n+\ell}=\mathbb{R}^{n+\ell}\times0\rightarrow
\mathbb{R}^{n+\ell+m}$. Then we may regard $\nu_{N}\oplus\varepsilon_{N}^{m}%
$\ as the normal bundle of $e_{N}^{m}$ and the Pontrjagin-Thom construction
for $e_{N}^{m}$ is the $m$-th suspension $S^{n+\ell+m}\rightarrow T(\nu
_{N}\oplus\varepsilon_{N}^{m})=T(\nu_{N})\wedge S^{m}$ of that for $e_{N}$.

Given two embeddings $e_{N_{i}}:N_{i}\rightarrow\mathbb{R}^{n+\ell_{i}}$
(\thinspace$i=0,1$) with $\ell_{1}=\ell_{0}+m$ for $m\geqq0$, by the above
argument we consider $e_{N_{0}}^{m}$ in place of $e_{N_{0}}$. It will be
proved by the argument below that $\omega([f])$ does not depend on the choice
of $e_{N_{0}}^{m}$ and $e_{N_{1}}$.

Let $\epsilon$ be a sufficiently small positive real number. Let
$I(0,\epsilon)$\ and $I(1,\epsilon)$\ denote the intervals $[0,\epsilon]$ and
$[1-\epsilon,1]$ respectively. Let $\mathfrak{C}:(W,\partial W)\rightarrow
(P\times\lbrack0,1],P\times0\cup P\times1)$ be an $\Omega_{\star}$-cobordism
of $f_{0}$ and $f_{1}$. Take embeddings $e_{N_{i}}:N_{i}\rightarrow
\mathbb{R}^{n+\ell}$ and $e_{W}:W\rightarrow\mathbb{R}^{n+\ell}\times
\lbrack0,1]$, and let us identify as $N_{i}=e_{N_{i}}(N_{i})=e_{N_{i}}%
(N_{i})\times\{i\}$, $W=e_{W}(W)$ and $P=P\times\{i\}$. Then we may assume
that for $i=0,1$,

(i) $W\cap(S^{n+\ell}\times I(i,\epsilon))=N_{i}\times I(i,\epsilon),$

(ii) $e_{W}|N_{i}\times I(i,\epsilon)=e_{N_{i}}\times id_{I(i,\epsilon)}$,

(iii) $\mathfrak{C}|N_{i}\times I(i,\epsilon)=f_{i}\times id_{I(i,\epsilon)},$

(iv) $j^{k}\mathfrak{C}|N_{i}\times\{i\}=j^{k}(f_{i}\times id_{\mathbb{R}%
})|N_{i}\times\{i\}$ under $N_{i}=N_{i}\times\{i\}$ and $P=P\times\{i\}.$

\noindent In the identification $TW|_{N_{i}}=TN_{i}\oplus\varepsilon_{N_{i}%
}^{1}$, the positive direction of $\varepsilon_{N_{0}}^{1}$ should correspond
to the inward normal direction, and that of $\varepsilon_{N_{1}}^{1}$ should
correspond to the outward normal direction. Then we may assume that the
trivializations%
\[
t_{N_{i}}:TN_{i}\oplus\ \nu_{N_{i}}\rightarrow\varepsilon_{N_{i}}^{n+\ell
}\text{ \ \ and \ \ }t_{W}:TW\oplus\ \nu_{W}\rightarrow\varepsilon_{W}%
^{n+\ell+1}%
\]
satisfy $t_{W}|_{N_{i}}=(t_{N_{i}}\oplus id_{\varepsilon_{N_{i}}^{1}}%
)\circ(id_{TN_{i}}\oplus\mathbf{k}_{N_{i}}^{\backsim})$, where $\mathbf{k}%
_{N_{i}}^{\backsim}:\varepsilon_{N_{i}}^{1}\oplus\nu_{N_{i}}\rightarrow
\nu_{N_{i}}\oplus\varepsilon_{N_{i}}^{1}$ is the map interchanging the
components $\varepsilon_{N_{i}}^{1}$ and $\nu_{N_{i}}$. Let $c_{\gamma_{G_{n}%
}^{n}\oplus\varepsilon_{G_{n}}^{1}}:\gamma_{G_{n}}^{n}\oplus\varepsilon
_{G_{n}}^{1}\rightarrow\gamma_{G_{n+1}}^{n+1}$ denote the bundle map covering
$i^{G}:G_{n}\rightarrow G_{n+1}$ defined by $c_{\gamma_{G_{n}}^{n}%
\oplus\varepsilon_{G_{n}}^{1}}(V_{x}\oplus(x,t))=V_{i^{G}(x)}+t\mathbf{e}%
_{n+\ell+1}$ for $x\in G_{n}$, $V_{x}\in(\gamma_{G_{n}}^{n})_{x}$ and
$t\in\mathbb{R}$. Let $s_{N_{i}}=\mathbf{j}(c_{TN_{i}},id_{TP})_{\Omega}\circ
j^{k}f_{i}$\ and $s_{W}=\mathbf{j}(c_{TW},id_{T(P\times\lbrack0,1])}%
)_{\Omega_{\star}}\circ j^{k}\mathfrak{C}$. Then we have that%
\[%
\begin{array}
[c]{l}%
c_{\gamma_{G_{n}}^{n}\oplus\varepsilon_{G_{n}}^{1}}\circ(c_{TN_{i}}%
\oplus(c_{N_{i}}\times id_{\mathbb{R}}))=c_{TW}|TN_{i}\oplus\varepsilon
_{N_{i}}^{1},\\
c_{s_{W}}^{\nu_{W}}|_{N_{i}}=\mathbf{b}(\widehat{\gamma})^{(\Omega
,\Omega_{\star})}\circ c_{s_{N_{i}}}^{\nu_{N_{i}}}.
\end{array}
\]
Let $a_{W}:S^{n+\ell}\times\lbrack0,1]\rightarrow T(\nu_{W})$ be the
Pontrjagin-Thom construction for $e_{W}$. Under the identifications%
\begin{align*}
\Omega_{\star}(TW,T(P\times\lbrack0,1]))  &  =\Omega_{\star}(TW,TP\oplus
\varepsilon_{P}^{1})\times\lbrack0,1],\\
\Omega_{\star}(\gamma_{G_{n+1}}^{n+1},T(P\times\lbrack0,1]))  &
=\Omega_{\star}(\gamma_{G_{n+1}}^{n+1},TP\oplus\varepsilon_{P}^{1}%
)\times\lbrack0,1],
\end{align*}
the composite $T(c_{s_{W}}^{\nu_{W}})\circ a_{W}$ gives a homotopy between
$T(\mathbf{b}(\widehat{\gamma})^{(\Omega,\Omega_{\star})})\circ T(c_{s_{N_{0}%
}}^{\nu_{N_{0}}})\circ a_{N_{0}}$ and $T(\mathbf{b}(\widehat{\gamma}%
)^{(\Omega,\Omega_{\star})})\circ T(c_{s_{N_{1}}}^{\nu_{N_{1}}})\circ
a_{N_{1}}.$ This proves $\omega([f_{0}])=\omega([f_{1}])$.
\end{proof}

\section{$\mathfrak{n}^{(\Omega,\Omega_{\star})}$ and $\mathfrak{o}%
^{(\Omega,\Omega_{\star})}$ are isomorphisms}

We prove Theorem 1.1 in this section. Let $\Omega(n,p)$ and $\Omega_{\star
}(n+1,p+1)$ be the open subsets given in Theorem 1.1.

\begin{proof}
[Proof of Theorem 1.1]We use the notation in the proof of Lemma 3.1. We first
prove that $\omega$ is injective. For this, take two $\Omega$-regular maps
$f_{i}:N_{i}\rightarrow P$ $(i=0,1)$ such that $\omega([f_{0}])=\omega
([f_{1}])$. Recall the map $T(\mathbf{b}(\widehat{\gamma})^{(\Omega
,\Omega_{\star})})\circ T(c_{s_{N_{i}}}^{\nu_{N_{i}}})\circ a_{N_{i}}$ which
represents $\omega([f_{i}])$. There is a homotopy $H:S^{n+\ell}\times
\lbrack0,1]\rightarrow T(\widehat{\gamma}_{\boldsymbol{\Omega}_{\star}}^{\ell
})$ such that if $\epsilon$ is sufficiently small, then we have, for $i=0,1$,

(i) $H|S^{n+\ell}\times I(i,\epsilon)=T(\mathbf{b}(\widehat{\gamma}%
)^{(\Omega,\Omega_{\star})})\circ T(c_{s_{N_{i}}}^{\nu_{N_{i}}})\circ
a_{N_{i}}\circ(\pi_{S^{n+\ell}}|S^{n+\ell}\times I(i,\epsilon)),$

(ii) $H$ is smooth around $H^{-1}(\boldsymbol{\Omega}_{\star})$ and is
transverse to $\boldsymbol{\Omega}_{\star}$.

\noindent We set $W=H^{-1}(\boldsymbol{\Omega}_{\star})$. Then we have

(iii) $W\cap(S^{n+\ell}\times I(i,\epsilon))=N_{i}\times I(i,\epsilon)$,

(iv) $H|N_{i}\times I(i,\epsilon)=\Delta^{(\Omega,\Omega_{\star})}\circ
s_{N_{i}}\circ(\pi_{N_{i}}|(N_{i}\times I(i,\epsilon))$,

(v) $TW|_{N_{i}\times I(i,\epsilon)}=T(N_{i}\times I(i,\epsilon))=(TN_{i}%
\oplus\varepsilon_{N_{i}}^{1})\times I(i,\epsilon)$,

(vi) $\nu_{W}|_{N_{i}\times I(i,\epsilon)}=\nu_{N_{i}}\times I(i,\epsilon)$.

\noindent By (ii) we have the bundle map $c_{\thicksim}^{\nu_{W}}:\nu
_{W}\rightarrow\widehat{\gamma}_{\boldsymbol{\Omega}_{\star}}^{\ell}$ covering
$H|W:W\rightarrow\boldsymbol{\Omega}_{\star}$ such that

(vii) $c_{\thicksim}^{\nu_{W}}|_{N_{i}\times I(i,\epsilon)}=\mathbf{b}%
(\widehat{\gamma})^{(\Omega,\Omega_{\star})}\circ c_{s_{N_{i}}}^{\nu_{N_{i}}%
}\circ(\pi_{\nu_{N_{i}}}|\nu_{N_{i}}\times I(i,\epsilon))$ by (i) and (iv).

\noindent By \cite[Proposition 3.3]{FoldSurg} we obtain a bundle map of
$TW\oplus\varepsilon_{W}^{3}$\ to $(\pi_{G_{n+1}}^{k})^{\ast}(\gamma_{G_{n+1}%
}^{n+1}\oplus\varepsilon_{G_{n+1}}^{3})|_{\boldsymbol{\Omega}_{\star}}$\ with
a required property concerning the trivialization. By this property and the
dimensional reason, we obtain a bundle map%
\[
c_{\thicksim}^{TW}:TW\rightarrow(\pi_{G_{n+1}}^{k})^{\ast}(\gamma_{G_{n+1}%
}^{n+1})|_{\boldsymbol{\Omega}_{\star}}%
\]
covering $H|W:W\rightarrow\boldsymbol{\Omega}_{\star}$ induced from the above
bundle map such that $t_{\boldsymbol{\Omega}_{\star}}\circ(c_{\thicksim}%
^{TW}\oplus c_{\thicksim}^{\nu_{W}})\circ t_{W}^{-1}$ is homotopic to
$(H|W)\times id_{\mathbb{R}^{n+\ell+1}}$. Since $\gamma_{G_{n+1}}^{n+1}$ is
the universal bundle $(\ell\gg n)$, $c_{\thicksim}^{TW}$ is regarded as
$c_{H|W}^{TW}$. Let $\mathbf{b}(\gamma\oplus\varepsilon^{1})^{(\Omega
,\Omega_{\star})}:(\pi_{G_{n}}^{k})^{\ast}(\gamma_{G_{n}}^{n}\oplus
\varepsilon_{G_{n}}^{1})|_{\boldsymbol{\Omega}}\rightarrow(\pi_{G_{n+1}}%
^{k})^{\ast}(\gamma_{G_{n+1}}^{n+1})|_{\boldsymbol{\Omega}_{\star}}$ covering
$\Delta^{(\Omega,\Omega_{\star})}$ be the bundle map induced from
$c_{\gamma_{G_{n}}^{n}\oplus\varepsilon_{G_{n}}^{1}}$ in the proof of Lemma
3.1. Then we may assume by (iv), (v) and (vi) that

(viii) $c_{\thicksim}^{TW}|_{(N_{i}\times I(i,\epsilon))}$ is equal to%
\[
\mathbf{b}(\gamma\oplus\varepsilon^{1})^{(\Omega,\Omega_{\star})}%
\circ(c_{s_{N_{i}}}^{TN_{i}}\oplus(c_{s_{N_{i}}}\times id_{\mathbb{R}}%
))\circ\pi_{TN_{i}\oplus\varepsilon_{N_{i}}}|((TN_{i}\oplus\varepsilon_{N_{i}%
})\times I(i,\epsilon)).
\]
Hence, $c_{W}$ is homotopic to $\pi_{G_{n+1}}^{k}\circ H|W$ relative to
$(N_{0}\times\lbrack0,\epsilon])\cup(N_{1}\times\lbrack1-\epsilon,1])$. Let
$\pi_{TP\oplus\varepsilon_{P}^{1}}:T(P\times\lbrack0,1])=(TP\oplus
\varepsilon_{P}^{1})\times\lbrack0,1]\rightarrow TP\oplus\varepsilon_{P}^{1}$
be the canonical bundle map covering the canonical projection $\pi_{P}%
:P\times\lbrack0,1]\rightarrow P$. Then we have the bundle map%
\begin{align*}
\mathbf{j}(c_{TW},\pi_{TP\oplus\varepsilon_{P}^{1}})_{\Omega_{\star}}  &
:\Omega_{\star}(TW,T(P\times\lbrack0,1]))=\Omega_{\star}(TW,TP\oplus
\varepsilon_{P}^{1})\times\lbrack0,1]\\
&  \longrightarrow\boldsymbol{\Omega}_{\star}=\Omega_{\star}(\gamma_{G_{n+1}%
}^{n+1},TP\oplus\varepsilon_{P}^{1})
\end{align*}
covering $c_{W}\times\pi_{P}$. Therefore, since $[0,1]$ is contractible, there
is a section $s_{W}:W\rightarrow\Omega_{\star}(TW,T(P\times\lbrack0,1]))$ such
that%
\begin{align*}
\pi_{P\times\lbrack0,1]}^{k}\circ s_{W}|N_{i}\times I(i,\epsilon)  &
=f_{i}\times id_{I(i,\epsilon)},\\
\pi_{P}\circ\pi_{P\times\lbrack0,1]}^{k}\circ s_{W}  &  =\pi_{P}^{k}%
\circ(H|W),
\end{align*}
and that $\mathbf{j}(c_{TW},\pi_{TP\oplus\varepsilon_{P}^{1}})_{\Omega_{\star
}}\circ s_{W}$ is homotopic to $H|W$ relative to $(N_{0}\times\lbrack
0,\epsilon])\cup(N_{1}\times\lbrack1-\epsilon,1])$.

Since $\Omega_{\star}(TW,T(P\times\lbrack0,1]))$ satisfies the relative
h-principle on the existence level, there exists an $\Omega_{\star}$-regular
map $\mathfrak{C}:W\rightarrow P\times\lbrack0,1]$ such that $\mathfrak{C}%
(x,t)=f_{0}(x)\times t$ for $0\leqq t\leqq\epsilon$, $\mathfrak{C}%
(x,t)=f_{1}(x)\times t$ for $1-\epsilon\leqq t\leqq1$ and that $j^{k}%
\mathfrak{C}$ is homotopic to $s_{W}$ relative to $(N_{0}\times\lbrack
0,\epsilon/2])\cup(N_{1}\times\lbrack1-\epsilon/2,1])$. This implies that the
$\Omega$-regular maps $f_{0}$ and $f_{1}$ are $\Omega_{\star}$-cobordant. This
proves that $\omega$ is injective.

We next prove that $\omega$ is surjective. Let an element $\widetilde{\alpha}%
$\ of $\mathrm{\operatorname{Im}}\left(  T(\mathbf{b}(\widehat{\gamma
})^{(\Omega,\Omega_{\star})})\right)  $ be represented by a map $\alpha
:S^{n+\ell}\rightarrow T(\widehat{\gamma}_{\boldsymbol{\Omega}}^{\ell})$ such
that $(T(\mathbf{b}(\widehat{\gamma})^{(\Omega,\Omega_{\star})}))_{\ast
}([\alpha])=\widetilde{\alpha}$. We may suppose that $\alpha$ is smooth around
$\alpha^{-1}(\boldsymbol{\Omega})$ and is transverse to $\boldsymbol{\Omega}$.
We set $N=\alpha^{-1}(\boldsymbol{\Omega})$. If $N=\emptyset$, then $[\alpha]$
must be a null element, although we can deform $\alpha$\ so that
$N\neq\emptyset$\ even in this case.\ Since $\alpha$ is transverse to
$\boldsymbol{\Omega}$, we have the bundle map $c_{\thicksim}^{\nu_{N}}:\nu
_{N}\rightarrow\widehat{\gamma}_{\boldsymbol{\Omega}}^{\ell}$ covering
$\alpha|N$. It follows from \cite[Proposition 3.3]{FoldSurg} that there exists
a bundle map%
\[
c_{\thicksim}^{TN\oplus\varepsilon_{N}^{3}}:TN\oplus\varepsilon_{N}%
^{3}\rightarrow(\pi_{G_{n}}^{k})^{\ast}(\gamma_{G_{n}}^{n}\oplus
\varepsilon_{G_{n}}^{3})|_{\boldsymbol{\Omega}}%
\]
covering $\alpha|N:N\rightarrow\boldsymbol{\Omega}$ such that the composite%
\begin{align*}
&  (t_{\boldsymbol{\Omega}}\oplus id_{\varepsilon_{\boldsymbol{\Omega}}^{3}%
})\circ(id_{(\pi_{G_{n}}^{k})^{\ast}(\gamma_{G_{n}}^{n})}\oplus\mathbf{k}%
_{G_{n}}^{\backsim})\\
&  \text{ \ \ \ \ \ \ \ \ \ \ \ \ \ \ \ \ \ }\circ(c_{\thicksim}%
^{TN\oplus\varepsilon_{N}^{3}}\oplus c_{\thicksim}^{\nu_{N}})\circ
(id_{TN}\oplus\mathbf{k}_{N}^{\backsim})\circ(t_{N}^{-1}\oplus id_{\varepsilon
_{N}^{3}})
\end{align*}
is homotopic to $(\alpha|N)\times id_{\mathbb{R}^{n+\ell+3}}$, where
$\mathbf{k}_{G_{n}}^{\backsim}:\varepsilon_{G_{n}}^{3}\oplus\widehat{\gamma
}_{G_{n}}^{\ell}\rightarrow\widehat{\gamma}_{G_{n}}^{\ell}\oplus
\varepsilon_{G_{n}}^{3}$ and $\mathbf{k}_{N}^{\backsim}:\nu_{N}\oplus
\varepsilon_{N}^{3}\rightarrow\varepsilon_{N}^{3}\oplus\nu_{N}$\ are the maps
interchanging the components respectively. Since $\gamma_{G_{n}}^{n}$ is the
universal bundle $(\ell\gg n)$, $c_{\thicksim}^{TN\oplus\varepsilon_{N}^{3}}$
is homotopic to $c_{\alpha|N}^{TN}\oplus((\alpha|N)\times id_{\mathbb{R}^{3}%
})$, and $t_{\boldsymbol{\Omega}}\circ(c_{\alpha|N}^{TN}\oplus c_{\thicksim
}^{\nu_{N}})\circ t_{N}^{-1}$ is homotopic to $(\alpha|N)\times id_{\mathbb{R}%
^{n+\ell}}$. Hence, $c_{N}$ is homotopic to $\pi_{G_{n}}^{k}\circ\alpha|N$. By
\cite[Proposition 3.3]{FoldSurg} again, $c_{\thicksim}^{\nu_{N}}$ and
$c_{\alpha|N}^{\nu_{N}}$ are homotopic as bundle maps $\nu_{N}\rightarrow
\widehat{\gamma}_{\boldsymbol{\Omega}}^{\ell}$. Since we have the bundle map%
\[
\mathbf{j}(c_{TN},id_{TP})_{\boldsymbol{\Omega}}:\Omega(TN,TP)\longrightarrow
\boldsymbol{\Omega}=\Omega(\gamma_{G_{n}}^{n},TP)
\]
covering $c_{N}\times id_{P}$, there is a section $s_{N}:N\rightarrow
\Omega(TN,TP)$ such that $\pi_{P}^{k}\circ s_{N}=\pi_{P}^{k}\circ\alpha|N$ and
$\mathbf{j}(c_{TN},id_{TP})_{\boldsymbol{\Omega}}\circ s_{N}$ is homotopic to
$\alpha|N.$ Since $\Omega(TN,TP)$ satisfies the h-principle, there exists an
$\Omega$-regular map $f:N\rightarrow P$ such that $j^{k}f$ is homotopic to
$s_{N}$. This implies that $\mathbf{j}(c_{TN},id_{TP})_{\boldsymbol{\Omega}%
}\circ j^{k}f$ and $\alpha|N$ are homotopic. This proves that
\begin{align*}
\omega([f])  &  =[T(\mathbf{b}(\widehat{\gamma})^{(\Omega,\Omega_{\star}%
)})\circ T(c_{\alpha|N}^{\nu_{N}})\circ a_{N}]\\
&  =[T(\mathbf{b}(\widehat{\gamma})^{(\Omega,\Omega_{\star})})\circ
T(c_{\thicksim}^{\nu_{N}})\circ a_{N}]\\
&  =[T(\mathbf{b}(\widehat{\gamma})^{(\Omega,\Omega_{\star})})\circ\alpha]\\
&  =\widetilde{\alpha}.
\end{align*}
This is what we want.
\end{proof}

Under the assumption of Theorem 1.1 $\mathfrak{C}(n,P;\Omega,\Omega_{\star})$
inherits the structure of an abelian group from the stable homotopy groups.
The null element is defined to be represented by an $\Omega$-regular map
$f:N\rightarrow P$, which has an $\Omega_{\star}$-cobordism $\mathfrak{C}%
:(W,\partial W)\rightarrow(P\times\lbrack0,1],P\times0)$ with $\partial W=N$
such that $\mathfrak{C}|N=f$ under the identification $P\times0=P$.

\section{Examples of $\Omega_{\star}(n+1,p+1)$}

In this section $n<p$ is not necessarily assumed. As an important example of
$\Omega_{\ast}(n+1,p+1)$ for $\Omega(n,p)$ we recall $\Omega_{\mathcal{K}}%
^{1}=\Omega_{\mathcal{K}}(n+1,p+1)$, which is the set consisting of all
$\mathcal{K}$-orbits $\mathcal{K}(i_{+1}(z))$ for $k$-jets $z\in\Omega(n,p)$.
Let $\mathcal{C}_{m}$ denote the ring of smooth function germs $(\mathbb{R}%
^{m},0)\rightarrow\mathbb{R}$ and let $\mathfrak{m}_{m}$ denote its maximal ideal.

\begin{lemma}
Let $i$ be a non-negative integer smaller than $n+1$. Then any $k$-jet
$w\in\Sigma^{i}(n+1,p+1)$ has a $k$-jet $z\in\Sigma^{i}(n,p)$ such that $w$
lies in $\mathcal{K}(i_{+1}(z))$.
\end{lemma}

\begin{proof}
We consider the usual coordinates $x=(x_{1},x_{2},\cdots,x_{n+1})$ of
$\mathbb{R}^{n+1}$ and $y=(y_{1},y_{2},\cdots,y_{p+1})$ of $\mathbb{R}^{p+1}$.
Under suitable respective coordinates%
\[
x^{\prime}=(x_{1}^{\prime},x_{2}^{\prime},\cdots,x_{n+1}^{\prime})\text{ \ and
\ }y^{\prime}=(y_{1}^{\prime},y_{2}^{\prime},\cdots,y_{p+1}^{\prime})
\]
of $\mathbb{R}^{n+1}$ and $\mathbb{R}^{p+1}$, $w$ is represented as
$w=j_{0}^{k}g$ with%
\[
(y_{1}^{\prime}\circ g(x^{\prime}),\cdots,y_{p+1}^{\prime}\circ g(x^{\prime
}))=(x_{1}^{\prime},\cdots,x_{n-i}^{\prime},g^{n-i+1}(x^{\prime}),\cdots
,g^{p}(x^{\prime}),x_{n+1}^{\prime}),
\]
where $g^{j}\in\mathfrak{m}_{n+1}^{2}$ . Let $\overline{g}:\mathbb{R}%
^{n+1}\rightarrow\mathbb{R}^{p+1}$\ be the map germ defined by%
\[
(y_{1}\circ\overline{g}(x),\cdots,y_{p+1}\circ\overline{g}(x))=(x_{1}%
,\cdots,x_{n-i},g^{n-i+1}(x),\cdots,g^{p}(x),x_{n+1}).
\]
It is evident that $j_{0}^{k}\overline{g}\in\mathcal{K}(w)$. Let
$\overset{\circ}{x}=(x_{1},\cdots,x_{n},0)$. Define the map germ
$f:\mathbb{R}^{n}=\mathbb{R}^{n}\times0\rightarrow\mathbb{R}^{p}$\ by%
\[
f(\overset{\circ}{x})=\left\{
\begin{array}
[c]{ll}%
(x_{1},\cdots,x_{n-i},g^{n-i+1}(\overset{\circ}{x}),g^{n-i+2}(\overset{\circ
}{x}),\cdots,g^{p}(\overset{\circ}{x})) & \text{for }i<n,\\
(g^{1}(\overset{\circ}{x}),g^{2}(\overset{\circ}{x}),\cdots,g^{p}%
(\overset{\circ}{x})) & \text{for }i=n.
\end{array}
\right.
\]

Then we have%
\begin{align*}
Q(j_{0}^{k}\overline{g})  &  =\mathcal{C}_{n+1}/(\overline{g}^{\ast
}(\mathfrak{m}_{p+1})+\mathfrak{m}_{n+1}^{k+1})\\
&  \approx\mathcal{C}_{n}/(f^{\ast}(\mathfrak{m}_{p})+\mathfrak{m}_{n}%
^{k+1})\\
&  =Q(j_{0}^{k}f).
\end{align*}
Setting $z=j_{0}^{k}f$, we have $w\in\mathcal{K}(i_{+1}(z))$.
\end{proof}

\begin{lemma}
Let $\Omega(n,p)$ be a $\mathcal{K}$-invariant open subset of $J^{k}(n,p)$.
Then the set $\Omega_{\mathcal{K}}^{m}$ is open in $J^{k}(n+m,p+m)$ for $m>0$.
\end{lemma}

\begin{proof}
It is enough to prove the case $m=1$. Suppose to the contrary that
$\Omega_{\mathcal{K}}^{1}$ is not open. Then there exists a $k$-jet
$w\in\Omega_{\mathcal{K}}^{1}$ and $k$-jets $w_{j}\notin\Omega_{\mathcal{K}%
}^{1}$ such that $\lim_{j\rightarrow\infty}w_{j}=w$. By definition, there
exists a $k$-jet $z\in\Omega(n,p)$\ such that $w\in\mathcal{K}(i_{+1}(z))$.
Let $i=n-$rank$z$. Then we may assume without loss of generality that
$w=i_{+1}(z)$, $w=j_{0}^{k}g$, $w_{j}=j_{0}^{k}g_{j}$ and $\lim_{j\rightarrow
\infty}w_{j}=w$\ with%
\begin{align*}
(y_{1}\circ g(x),\cdots,y_{p+1}\circ g(x))  &  =(x_{1},\cdots,x_{n-i}%
,g^{n-i+1}(x),\cdots,g^{p}(x),x_{n+1}),\\
(y_{1}\circ g_{j}(x),\cdots,y_{p+1}\circ g_{j}(x))  &  =(g_{j}^{1}%
(x),\cdots,g_{j}^{n-i}(x),g_{j}^{n-i+1}(x),\cdots,g_{j}^{p+1}(x)),
\end{align*}
where $g^{t}\in\mathfrak{m}_{n}^{2}$ for $n-i+1\leqq t\leqq p$. Since
$\lim_{j\rightarrow\infty}w_{j}=i_{+1}(z)$, we set%
\begin{align*}
h_{j}^{1}(x_{1},\cdots,x_{n+1})  &  =(x_{1},\cdots,x_{n-i},g_{j}%
^{n-i+1}(x),\cdots,g_{j}^{p}(x),x_{n+1}),\\
h_{j}^{2}(x_{1},\cdots,x_{n+1})  &  =(x_{1},\cdots,x_{n-i},g_{j}%
^{n-i+1}(\overset{\circ}{x}),\cdots,g_{j}^{p}(\overset{\circ}{x}),x_{n+1}).
\end{align*}
Since the map germ defined by%
\[
(x_{1},\cdots,x_{n+1})\longmapsto(g_{j}^{1}(x),\cdots,g_{j}^{n-i}%
(x),x_{n-i+1},\cdots,x_{n},g_{j}^{p+1}(x))
\]
is a local diffeomorphism for sufficiently large numbers $j$, we have
$Q_{k}(w_{j})\approx Q_{k}(j_{0}^{k}h_{j}^{1})\approx Q_{k}(j_{0}^{k}h_{j}%
^{2})$.\ Let us define%
\[
f_{j}(x_{1},\cdots,x_{n})=\left\{
\begin{array}
[c]{ll}%
(x_{1},\cdots,x_{n-i},g_{j}^{n-i+1}(\overset{\circ}{x}),\cdots,g_{j}%
^{p}(\overset{\circ}{x})) & \text{for }i<n,\\
(g_{j}^{1}(\overset{\circ}{x}),\cdots,g_{j}^{p}(\overset{\circ}{x})) &
\text{for }i=n.
\end{array}
\right.
\]
Then we have that $\lim_{j\rightarrow\infty}j_{0}^{k}h_{j}^{1}=w=i_{+1}(z)$
and $i_{+1}(j_{0}^{k}f_{j})=j_{0}^{k}h_{j}^{2}$, and hence, $Q_{k}%
(w_{j})\approx Q_{k}(j_{0}^{k}f_{j})$. Since $\lim_{j\rightarrow\infty}%
j_{0}^{k}f_{j}=z$, we have that $j_{0}^{k}f_{j}\in\Omega(n,p)$\ for
sufficiently large numbers $j$. By definition, we have $j_{0}^{k}h_{j}^{2}%
\in\Omega_{\mathcal{K}}^{1}$ for sufficiently large numbers $j$. Since
$Q_{k}(w_{j})\approx Q_{k}(j_{0}^{k}h_{j}^{2})$, it follows from \cite[Theorem
2.1]{MathIV} that $w_{j}\in\Omega_{\mathcal{K}}^{1}$ for sufficiently large
numbers $j$.\ This is a contradiction.
\end{proof}

\begin{lemma}
If two map germs $f_{1}$, $f_{2}:(\mathbb{R}^{m},0)\rightarrow(\mathbb{R}%
^{q},0)$ are $\mathcal{K}$-equivalent, then the Boardman symbols of $j_{0}%
^{k}f_{1}$ and $j_{0}^{k}f_{2}$ are the same. Consequently, the Boardman
manifold $\Sigma^{I}(m,q)$ is invariant with respect to the action of
$\mathcal{K}$.
\end{lemma}

\begin{proof}
By \cite{Mar1} there exist a germ of a diffeomorphism $h:(\mathbb{R}%
^{m},0)\rightarrow(\mathbb{R}^{m},0)$ and a smooth map germ $M:(\mathbb{R}%
^{m},0)\rightarrow GL(q)$ such that $M(x)f_{1}(h(x))=f_{2}(x)$. Let
$f_{i}(x)=(f_{1}^{i}(x),\cdots,f_{q}^{i}(x))$ with $f_{j}^{i}\in
\mathcal{C}_{m}/\mathfrak{m}_{m}^{k+1}$ ($i=1,2$). Let $\mathfrak{I}(f_{i})$
denote the ideal of $\mathcal{C}_{m}/\mathfrak{m}_{m}^{k+1}$ generated by
$f_{1}^{i},\cdots,f_{q}^{i}$. Let $h_{\ast}:\mathcal{C}_{m}/\mathfrak{m}%
_{m}^{k+1}\rightarrow\mathcal{C}_{m}/\mathfrak{m}_{m}^{k+1}$ be the
isomorphism defined by $h_{\ast}(\phi)=\phi\circ h$. Then we have $h_{\ast
}(\mathfrak{I}(f_{1}))=\mathfrak{I}(f_{2})$. The Boardman symbols of
$j_{0}^{k}f_{i}$ ($i=1,2$) are determined by $\mathfrak{I}(f_{i})$ and are the
same by \cite{MathTB}. This proves the assertion.
\end{proof}

\begin{lemma}
If $I$ is a Boardman symbol such that $\Omega^{I}(n,p)$ is nonempty, then
$\Omega^{I}(n+1,p+1)$ is the union of all $\mathcal{K}$-orbits $\mathcal{K}%
(i_{+1}(z))$ for $z\in\Omega^{I}(n,p)$.
\end{lemma}

\begin{proof}
Let $I=(i_{1},i_{2},\cdots)$. Since $\Omega^{I}(n,p)$ is nonempty, we have
$i_{1}\leqq n$. Let $w\in\Omega^{I}(n+1,p+1)$, whose Boardman symbol is
$J\leqq I$. Since $j_{1}\leqq i_{1}\leqq n$, it follows from Lemma 5.1 that
there exists a $z\in\Omega^{I}(n,p)$\ such that $w\in\mathcal{K(}i_{+1}(z))$.

Conversely, let\ $z\in\Omega^{I}(n,p)$ with Boardman symbol $K\leqq I$. Since
the Boardman symbol of $i_{+1}(z)$ is obviously equal to $K$, we
have\ $i_{+1}(z)\in\Omega^{I}(n+1,p+1)$. This shows the assertion.
\end{proof}

\section{Preliminaries for Theorem 1.2}

We prepare lemmas and propositions for the proof of Theorem 1.2.

Let $\mathbf{e}_{i}=(0,\cdots,0,1,0,\cdots,0)$ with $1$ being the $i$-th
component. Let $pr_{p+1}:\mathbb{R}^{p+1}\rightarrow\mathbb{R}$ be the
projection mapping $(x_{1},\cdots,x_{p+1})$ to $x_{p+1}$. Let $\pi_{1}%
^{k}:J^{k}(n,p)\rightarrow J^{1}(n,p)$\ be the canonical forgetting projection.

Let $K$ be a finite simplicial complex and$\ L$ be its subcomplex such that
$K\backslash L$ is a manifold and $\dim L<\dim K$.

\begin{lemma}
Let $\Omega(n,p)$\ be a $\mathcal{K}$-invariant open subset of\ $J^{k}(n,p)$.
Let $(K,L)$ be given as above and $\dim K<p$. Let $\psi:(K,L)\rightarrow
(\Omega_{\mathcal{K}}^{1},i_{+1}(\Omega(n,p)))$ be a map such that
$\psi|(K\backslash L)$ is smooth. Then there exists a homotopy $\psi_{\lambda
}:(K,L)\rightarrow(\Omega_{\mathcal{K}}^{1},i_{+1}(\Omega(n,p)))$ such that

(i) $\psi_{0}=\psi,$

(ii) $\psi_{\lambda}|L=\psi|L,$

(iii) $pr_{p+1}(\pi_{1}^{k}\circ\psi_{1}(u)(x_{1},\cdots,x_{n+1}))=x_{n+1}$
for any $u\in K.$
\end{lemma}

\begin{proof}
Let us define $\mathbf{e}:K\rightarrow\mathbb{R}^{p+1}$ by $\mathbf{e}%
(u)=(\pi_{1}^{k}\circ\psi(u))(\mathbf{e}_{n+1})$. Since $\psi(L)\subset
i_{+1}(\Omega(n,p))$, we have that, for any $u\in L$, $\mathbf{e}%
(u)=\mathbf{e}_{p+1}$. Consider the fiber bundle $\mathfrak{d}_{\mathbb{R}%
^{p+1}}:J^{k}(n+1,p+1)\rightarrow\mathbb{R}^{p+1}$ defined by $\mathfrak{d}%
_{\mathbb{R}^{p+1}}(j_{\mathbf{0}}^{k}f)=j_{\mathbf{0}}^{1}f(\mathbf{e}%
_{n+1})$. Since $\dim K<p$, $K\backslash L$ is a manifold, $\psi|(K\backslash
L)$ is smooth and since $\Omega_{\mathcal{K}}^{1}$ is an open subset, it
follows from the transversality theorem and the covering homotopy property of
$\mathfrak{d}_{\mathbb{R}^{p+1}}$ that there exists a homotopy $\varphi
_{\lambda}:K\rightarrow J^{k}(n+1,p+1)$ relative to $L$ with $\varphi_{0}%
=\psi$\ such that

(6.1-i) the deformation $\mathbf{u}_{\lambda}=\mathfrak{d}_{\mathbb{R}^{p+1}%
}\circ\varphi_{\lambda}$ of $\mathbf{e}$ with $\mathbf{u}_{0}=\mathbf{e}%
$\ satisfies that $\mathbf{u}_{1}$ does not take the value of any non-positive
multiple of $\mathbf{e}_{p+1}$ and

(6.1-ii) $\varphi_{\lambda}(K)\subset\Omega_{\mathcal{K}}^{1}$ for any
$\lambda$.

In the following an element of $GL(m)$ is regarded as a linear isomorphism of
$\mathbb{R}^{m}$ and $E_{m}$\ is the unit matrix of degree $m$. Let
$h_{\lambda}^{1}:(K,L)\rightarrow(GL(p+1),E_{p+1})$ be the homotopy defined by
$h_{\lambda}^{1}(u)=((1-\lambda)+\lambda/\Vert\mathbf{u}_{1}(u)\Vert)E_{p+1}$.
It follows from (6.1-i) that $h_{1}^{1}(u)(\mathbf{u}_{1}(u))\in S^{p}$ and
$h_{1}^{1}(u)(\mathbf{u}_{1}(u))\neq-\mathbf{e}_{p+1}$ for any $u\in K$. By
considering the rotation which is the identity on all points orthogonal to
both $\mathbf{u}_{1}(u)$ and $\mathbf{e}_{p+1}$ and rotates the great circle
through $\mathbf{u}_{1}(u)$ and $\mathbf{e}_{p+1}$ so as to carry
$\mathbf{u}_{1}(u)$ to $\mathbf{e}_{p+1}$ along the shorter way (when
$\mathbf{u}_{1}(u)=\mathbf{e}_{p+1}$, we consider $E_{p+1}$), we have the
homotopy $h_{\lambda}^{2}:(K,L)\rightarrow(SO(p+1),E_{p+1})$ relative to $L$
such that $h_{0}^{2}(u)=E_{p+1}$ and $h_{1}^{2}(u)\circ h_{1}^{1}%
(u)(\mathbf{u}_{1}(u))=\mathbf{e}_{p+1}$ for any $u\in K$. Let $h_{\lambda
}:(K,L)\rightarrow(GL(p+1),E_{p+1})$ be the homotopy defined by $h_{\lambda
}=h_{2\lambda}^{1}$ for $0\leqq\lambda\leqq1/2$ and $h_{\lambda}%
=h_{2\lambda-1}^{2}\circ h_{1}^{1}$ for $1/2\leqq\lambda\leqq1$. Define
$\kappa_{p+1}:K\rightarrow J^{1}(n+1,1)$ by
\[
\kappa_{p+1}(u)=pr_{p+1}\circ\pi_{1}^{k}(j^{k}(h_{1}(u))\circ\varphi_{1}(u)).
\]
Since $\kappa_{p+1}(u)$ is of rank $1$ for any $u\in K$, we have the unique
vector $V(u)\in\mathbb{R}^{n+1}$ of length $1$ such that $V(u)$ is
perpendicular to $\mathrm{Ker}(\kappa_{p+1}(u))$ and $\kappa_{p+1}(u)(V(u))$
is positive. Namely, $h_{1}(u)\circ\pi_{1}^{k}(\varphi_{1}(u))(V(u))$ is
directed to the same orientation of $\mathbf{e}_{p+1}$. Since $\kappa
_{p+1}(u)(\mathbf{e}_{n+1})=1$, $V(u)$ cannot be equal to $-\mathbf{e}_{n+1}$.

We set $v(u)=\kappa_{p+1}(u)(V(u))>0$. By considering the rotation which is
the identity on all points orthogonal to both $V(u)$ and $\mathbf{e}_{n+1}$
and rotates the great circle through $V(u)$ and $\mathbf{e}_{n+1}$ so as to
carry $\mathbf{e}_{n+1}$ to $V(u)$ along the shorter way, we again have the
homotopy $H_{\lambda}^{1}:(K,L)\rightarrow(SO(n+1),E_{n+1})$ relative to $L$
such that $H_{0}^{1}(u)=E_{n+1}$ and $H_{1}^{1}(u)(\mathbf{e}_{n+1})=V(u)$ for
any $u\in K$. Let$\ H_{\lambda}^{2}:(K,L)\rightarrow(GL(n+1),E_{n+1})$ be the
homotopy relative to $L$ defined by $H_{\lambda}^{2}(u)=((1-\lambda
)+\lambda/v(u))E_{n+1}$. Let $H_{\lambda}:(K,L)\rightarrow(GL(n+1),E_{n+1})$
be the homotopy defined by $H_{\lambda}(u)=H_{2\lambda}^{1}(u)$ for
$0\leqq\lambda\leqq1/2$ and $H_{\lambda}(u)=H_{2\lambda-1}^{2}(u)\circ
H_{1}^{1}(u)$ for $1/2\leqq\lambda\leqq1$. Then we have that for any $u\in
K$,
\begin{align*}
\kappa_{p+1}(u)\circ H_{1}(u)(\mathbf{e}_{n+1})  &  =\kappa_{p+1}(u)\circ
H_{1}^{2}(u)\circ H_{1}^{1}(u)(\mathbf{e}_{n+1})\\
&  =\kappa_{p+1}(u)\circ H_{1}^{2}(u)(V(u))\\
&  =\kappa_{p+1}(u)(V(u))/v(u)\\
&  =1.
\end{align*}
Since $H_{1}^{1}(u)\in SO(n+1)$ and $\mathbf{e}_{i}$ is orthogonal to
$\mathbf{e}_{n+1}$\ ($i<n+1$), $H_{1}^{1}(u)(\mathbf{e}_{i})$ is orthogonal to
$H_{1}^{1}(u)(\mathbf{e}_{n+1})=V(u)$. Namely, $H_{1}(u)(\mathbf{e}_{i})$ lies
in $\mathrm{Ker}(\kappa_{p+1}(u))$. Hence, we have%
\[
\kappa_{p+1}(u)\circ H_{1}(u)(\mathbf{e}_{i})=0\text{ \ \ \ for }i<n+1.
\]
Define the homotopy $\psi_{\lambda}:(K,L)\rightarrow(\Omega_{\mathcal{K}}%
^{1},i_{+1}(\Omega(n,p))$ relative to $L$ by
\[
\psi_{\lambda}(u)=\left\{
\begin{array}
[c]{ll}%
\varphi_{3\lambda}(u) & \text{for }0\leqq\lambda\leqq1/3,\\
h_{3\lambda-1}(u)\circ\varphi_{1}(u) & \text{for }1/3\leqq\lambda\leqq2/3,\\
h_{1}(u)\circ\varphi_{1}(u)\circ H_{3\lambda-2}(u) & \text{for\ }%
2/3\leqq\lambda\leqq1.
\end{array}
\right.
\]
By the definition we have
\[
pr_{p+1}\circ\pi_{1}^{k}(\psi_{1}(u))(\mathbf{e}_{i})=\left\{
\begin{array}
[c]{ll}%
0 & \text{for }i<n+1,\\
1 & \text{for }i=n+1.
\end{array}
\right.
\]

This is what we want.
\end{proof}

\begin{proposition}
Under the same assumption of Lemma 6.1, we have a homotopy $\Psi_{\lambda
}:(K,L)\rightarrow(\Omega_{\mathcal{K}}^{1},i_{+1}(\Omega(n,p)))$ such that

(i) $\Psi_{0}=\psi,$

(ii) $\Psi_{\lambda}|L=\psi|L,$

(iii) $\Psi_{1}(K)\subset i_{+1}(\Omega(n,p)).$
\end{proposition}

\begin{proof}
Let $\psi_{\lambda}$ be the homotopy given in Lemma 6.1. Let us express
$\psi_{\lambda}(u)=(f_{\lambda}^{1}(u),f_{\lambda}^{2}(u),\cdots,f_{\lambda
}^{p+1}(u))$\ using the coordinates of $\mathbb{R}^{p+1}$, where $f_{\lambda
}^{i}(u)$\ is regarded as a polynomial of degree at most $k$ with constant
$0$. We note that%
\[
f_{1}^{p+1}(u)(x_{1},\cdots,x_{n+1})=x_{n+1}+\text{higher term}.
\]
Let $\mathrm{Diff}(\mathbb{R}^{n+1},0)$\ be the space of all germs of local
diffeomorphisms of $(\mathbb{R}^{n+1},0)$. Let us define a homotopy of maps
$\Phi_{\lambda}:(K,L)\rightarrow(\mathrm{Diff}(\mathbb{R}^{n+1}%
,0),id_{(\mathbb{R}^{n+1},0)})$ by%
\[
\Phi_{\lambda}(u)(x_{1},\cdots,x_{n+1})=(x_{1},\cdots,x_{n+1}+\lambda
(f_{1}^{p+1}(u)-x_{n+1})).
\]
It is obvious that $\Phi_{\lambda}(u)$\ is a germ of a diffeomorphism of
$(\mathbb{R}^{n+1},0)$. Then we have the inverse $\Phi_{\lambda}(u)^{-1}$ such
that%
\begin{equation}
pr_{p+1}\circ\psi_{1}(u)\circ\Phi_{1}(u)^{-1}(x_{1},\cdots,x_{n+1}%
)=x_{n+1}\text{.}%
\end{equation}
We now define $\phi_{\lambda}:(K,L)\rightarrow(\Omega_{\mathcal{K}}^{1}%
,i_{+1}(\Omega(n,p)))$ by%
\[
\phi_{\lambda}(u)=\psi_{1}\circ j_{0}^{k}(\Phi_{\lambda}(u)^{-1}).
\]

In order to exclude the terms containing $x_{n+1}$ from $y_{j}\circ\phi_{1}%
$\ ($1\leqq j\leqq p$) we define the homotopy $\eta_{\lambda}:(K,L)\rightarrow
(J^{k}(n+1,p+1),i_{+1}(J^{k}(n,p)))$\ by
\begin{align}
\eta_{\lambda}(u)(x)  &  =(1-\lambda)\phi_{1}(u)(x_{1},\cdots,x_{n+1}%
)\nonumber\\
&  +\lambda(\phi_{1}(u)(x_{1},\cdots,x_{n},0)+(0,\cdots,0,x_{n+1})).
\end{align}
It is obvious that $\eta_{1}(K)\subset i_{+1}(\Omega(n,p))$ and that
$\eta_{\lambda}|L=\psi|L.$ It remains to prove that $\eta_{\lambda}$ is a
homotopy to $\Omega_{\mathcal{K}}^{1}$. It follows from (6.1) and (6.2) that%
\[
pr_{p+1}\circ\eta_{\lambda}(u)(x)=(1-\lambda)(x_{n+1})+\lambda x_{n+1}%
=x_{n+1}.
\]
Let us express $\eta_{\lambda}(u)=(g_{\lambda}^{1}(u),g_{\lambda}%
^{2}(u),\cdots,g_{\lambda}^{p+1}(u))$, where $g_{\lambda}^{i}(u)$\ is regarded
as a polynomial of degree at most $k$ with constant $0$.\ Consider the ideal
$\mathfrak{I}_{\lambda}(u)$\ generated by $g_{\lambda}^{1}(u),g_{\lambda}%
^{2}(u),\cdots,g_{\lambda}^{p+1}(u)$\ in $\mathfrak{m}_{n+1}\mathfrak{/m}%
_{n+1}^{k+1}$. Then $\mathfrak{I}_{\lambda}(u)$\ is constantly equal to
$\mathfrak{I}_{0}(u)$, and hence $Q(\eta_{\lambda}(u))\approx Q(\psi_{1}(u))$.
Since $\psi_{1}(u)\in\Omega_{\mathcal{K}}^{1}$, we have $\eta_{\lambda}%
(u)\in\Omega_{\mathcal{K}}^{1}$ by \cite{MathIV}. Then the required homotopy
$\Psi_{\lambda}$ is defined by $\Psi_{\lambda}=\psi_{3\lambda}\,$%
($0\leqq\lambda\leqq1/3$), $\Psi_{\lambda}=\phi_{3\lambda-1}\ $($1/3\leqq
\lambda\leqq2/3$) and $\Psi_{\lambda}=\eta_{3\lambda-2}$ ($2/3\leqq
\lambda\leqq1$).
\end{proof}

\begin{proposition}
Let $\Omega(n,p)$ be a $\mathcal{K}$-invariant open subset of $J^{k}(n,p)$.
Then $i_{+1}:\Omega(n,p)\rightarrow\Omega_{\mathcal{K}}^{1}$ is a homotopy
$(p-1)$-equivalence.
\end{proposition}

\begin{proof}
Let $\iota_{n}$ denote a jet of $\Omega(n,p)$ and $\iota_{n+1}=i_{+1}%
(\iota_{n})$. We first prove that $(i_{+1})_{\ast}:\pi_{i}(\Omega
(n,p))\rightarrow\pi_{i}(\Omega_{\mathcal{K}}^{1})$ is surjective for $0\leqq
i\leqq p-1$. Indeed, let $[a]\in\pi_{i}(\Omega_{\mathcal{K}}^{1})$\ be
represented by $a:(S^{i},\mathbf{e}_{1})\rightarrow(\Omega_{\mathcal{K}}%
^{1},\iota_{n+1})$. Then by Proposition 6.2 we have a homotopy $\varphi
_{\lambda}:(S^{i},\mathbf{e}_{1})\rightarrow(\Omega_{\mathcal{K}}^{1}%
,\iota_{n+1})$\ such that $\varphi_{1}(S^{i})\subset i_{+1}(\Omega(n,p))$.

Next let $0\leqq i<p-1$. Let $[b]\in\pi_{i}(\Omega(n,p))$\ be represented by
$b:(S^{i},\mathbf{e}_{1})\rightarrow(\Omega(n,p),\iota_{n})$ such that
$(i_{+1})_{\ast}([b])=0$. Then we have a homotopy $\widetilde{\varphi}%
:S^{i}\times\lbrack0,1]\rightarrow(\Omega_{\mathcal{K}}^{1},\iota_{n+1}%
)$\ such that $\widetilde{\varphi}|S^{i}\times0=i_{+1}\circ b$\ under the
identification $S^{i}=S^{i}\times0$ and $\widetilde{\varphi}(\mathbf{e}%
_{1}\times\lbrack0,1]\cup S^{i}\times1)=\iota_{n+1}$. It follows from
Proposition 6.2\ that since $i+1<p$, there exists a homotopy $\Phi_{\lambda
}:(S^{i}\times\lbrack0,1],S^{i}\times0\cup\mathbf{e}_{1}\times\lbrack0,1]\cup
S^{i}\times1)\rightarrow(\Omega_{\mathcal{K}}^{1},i_{+1}(\Omega(n,p)))$%
\ relative to $S^{i}\times0\cup\mathbf{e}_{1}\times\lbrack0,1]\cup S^{i}%
\times1$\ such that $\Phi_{1}(S^{i}\times\lbrack0,1])\subset i_{+1}%
(\Omega(n,p))$. This proves the injectivity of $(i_{+1})_{\ast}:\pi_{i}%
(\Omega(n,p))\rightarrow\pi_{i}(\Omega_{\mathcal{K}}^{1})$.
\end{proof}

\section{Proof of Theorem 1.2}

Let $A_{m,q}$ express $G_{m,q}$\ or $\widetilde{G}_{m,q}$. Let $i_{A_{m+r}%
}:A_{m,q}\rightarrow A_{m+r,q}$ denote\ the injection mapping an $m$-plane $a$
to the $(m+r)$-plane including $a$ and the canonical vectors $\mathbf{e}%
_{q+m+1},\cdots,\mathbf{e}_{q+m+r}$ in $\mathbb{R}^{q+m+r}$. We use the
notation $\boldsymbol{\Omega}_{\mathcal{K}}^{m}=\Omega_{\mathcal{K}}%
(\gamma_{G_{n+m}}^{n+m},TP\oplus\varepsilon_{P}^{m})$.

\begin{lemma}
The map $i_{A_{m+r}}:A_{m,q}\rightarrow A_{m+r,q}$ is a homotopy $m$-equivalence.
\end{lemma}

\begin{proof}
We only prove the unoriented case. The proof in the oriented case is similar.
Let us consider the diagram with the canonical maps as described%
\[%
\begin{array}
[c]{ccc}%
A_{m,q} & \overset{\rho_{0}}{\longrightarrow} & O(q+m+r)/O(q)\times O(m)\times
E_{r}\\
i_{A_{m+r}}\downarrow\text{ \ \ \ \ } & {\scriptsize \rho}_{2}\swarrow\text{
\ } & \downarrow{\scriptsize \rho}_{1}\\
A_{m+r,q} &  & O(q+m+r)/O(q+m)\times E_{r},
\end{array}
\]
where $\rho_{0}$ is induced from the inclusion $\mathbb{R}^{q+m}%
=\mathbb{R}^{q+m}\times0\rightarrow\mathbb{R}^{q+m+r}$ and $\rho_{1}$ and
$\rho_{2}$ are induced from the inclusions $O(q)\times O(m)\rightarrow O(q+m)$
and $O(m)\times E_{r}\rightarrow O(m+r)$ respectively. Since $A_{m,q}$ is a
fiber of the fiber bundle $\rho_{1}$, $\rho_{0}$ is a homotopy $(q+m-1)$%
-equivalence. Since $\rho_{2}$ is a homotopy $m$-equivalence, $i_{A_{m+r}}$ is
also a homotopy $m$-equivalence.
\end{proof}

In the following lemma $p$ is not necessarily larger than $n$.

\begin{lemma}
Let $\Omega(n,p)$ be a $\mathcal{K}$-invariant open subset of $J^{k}(n,p)$.
Then the fiber map $\Delta^{(\Omega,\Omega_{\mathcal{K}})}:\boldsymbol{\Omega
}\rightarrow\boldsymbol{\Omega}_{\mathcal{K}}^{1}$\ is a homotopy
$\min\{n,p-1\}$-equivalence.
\end{lemma}

\begin{proof}
Consider the commutative diagram%
\[%
\begin{array}
[c]{ccccccc}%
\longrightarrow & \pi_{i}(\Omega(n,p)) & \overset{\underrightarrow{\partial}%
}{} & \pi_{i}(\boldsymbol{\Omega}) & \longrightarrow & \pi_{i}(G_{n,\ell
}\times P) & \longrightarrow\\
& \downarrow &  & \downarrow &  & \downarrow & \\
\longrightarrow & \pi_{i}(\Omega_{\mathcal{K}}^{1}) & \overset
{\underrightarrow{\partial}}{} & \pi_{i}(\boldsymbol{\Omega}_{\mathcal{K}}%
^{1}) & \longrightarrow & \pi_{i}(G_{n+1,\ell}\times P) & \longrightarrow
\end{array}
\]
which is induced from the homomorphisms $(\Delta^{(\Omega,\Omega_{\mathcal{K}%
})})_{\ast}$\ of the exact sequence of\ the homotopy groups for the fiber
bundle $\boldsymbol{\Omega}$\ over $G_{n,\ell}\times P$\ to that for the fiber
bundle $\boldsymbol{\Omega}_{\mathcal{K}}^{1}$\ over $G_{n+1,\ell}\times P$.
Then it follows from Lemma 7.1 for $(i^{G})_{\ast}$ and Proposition 6.3 that
if $0\leqq i<\min\{n,p-1\}$, then $(\Delta^{(\Omega,\Omega_{\mathcal{K}}%
)})_{\ast}:\pi_{i}(\boldsymbol{\Omega})\rightarrow\pi_{i}(\boldsymbol{\Omega
}_{\mathcal{K}}^{1})$ is an isomorphism by the five lemma\ and if
$i=\min\{n,p-1\}$, then it is an epimorphism by \cite[Lemma 3.2]{CaEilen}.
\end{proof}

We are now ready to prove Theorem 1.2.

\begin{proof}
[Proof of Theorem 1.2]We only prove the unoriented case. In the oriented case
we only need the Thom Isomorphism Theorem under the coefficient group
$\mathbb{Z}$. It follows from Lemma 7.2 and the Whitehead theorem
(\cite[Section 5, 9 Theorem]{SpaAT}) that%
\[
(\Delta^{(\Omega,\Omega_{\mathcal{K}})})_{\ast}:H_{i}(\boldsymbol{\Omega
})\rightarrow H_{i}(\boldsymbol{\Omega}_{\mathcal{K}}^{1})
\]
is an isomorphism for $0\leqq i<n$ and an epimorphism for $i=n$.

Let $\ell\gg n,p$.\ By virtue of the Thom Isomorphism Theorem, we have that%
\[
T(\mathbf{b}(\widehat{\gamma})^{(\Omega,\Omega_{\mathcal{K}})})_{\ast
}:H_{i+\ell}\left(  T(\widehat{\gamma}_{\boldsymbol{\Omega}}^{\ell
});\mathbb{Z}/(2)\right)  \longrightarrow H_{i+\ell}\left(  T(\widehat{\gamma
}_{\boldsymbol{\Omega}_{\mathcal{K}}^{1}}^{\ell});\mathbb{Z}/(2)\right)
\]
is an isomorphism for $-\ell\leqq i<n$ and an epimorphism for $i=n$.

Let $\mathcal{C}$ denote the Serre class of finite groups of orders prime to
two. Then it follows from \cite[Proposition 2, p. 277]{Serre} that%
\begin{equation}
T(\mathbf{b}(\widehat{\gamma})^{(\Omega,\Omega_{\mathcal{K}})})_{\ast
}:H_{i+\ell}\left(  T(\widehat{\gamma}_{\boldsymbol{\Omega}}^{\ell
});\mathbb{Z}\right)  \longrightarrow H_{i+\ell}\left(  T(\widehat{\gamma
}_{\boldsymbol{\Omega}_{\mathcal{K}}^{1}}^{\ell});\mathbb{Z}\right)
\end{equation}
is a $\mathcal{C}$-isomorphism for $-\ell\leqq i<n$ and a $\mathcal{C}%
$-epimorphism for $i=n$.\ By the Whitehead theorem modulo $\mathcal{C}$
(\cite[Theorem 3, p. 276]{Serre}),
\[
T(\mathbf{b}(\widehat{\gamma})^{(\Omega,\Omega_{\mathcal{K}})})_{\ast}%
:\pi_{i+\ell}\left(  T(\widehat{\gamma}_{\boldsymbol{\Omega}}^{\ell})\right)
\longrightarrow\pi_{i+\ell}\left(  T(\widehat{\gamma}_{\boldsymbol{\Omega
}_{\mathcal{K}}^{1}}^{\ell})\right)
\]
is a $\mathcal{C}$-isomorphism for $-\ell\leqq i<n$ and a $\mathcal{C}%
$-epimorphism for $i=n$.

Let $\mathfrak{N}_{\boldsymbol{\Omega}}^{i}$\ (respectively $\mathfrak{N}%
_{\boldsymbol{\Omega}_{\mathcal{K}}^{1}}^{i}$) denote the cobordism group of
smooth maps $s$\ of closed $i$-manifolds $M$ to $\boldsymbol{\Omega}$
(respectively $\boldsymbol{\Omega}_{\mathcal{K}}^{1}$) under the corresponding
cobordism of smooth maps such that there exists a bundle map of the stable
$\ell$-dimensional normal bundle $\nu_{M}$ to $\widehat{\gamma}%
_{\boldsymbol{\Omega}}^{\ell}$ (respectively $\widehat{\gamma}%
_{\boldsymbol{\Omega}_{\mathcal{K}}^{1}}^{\ell}$) covering $s$. It follows
from the standard argument in the cobordism theory (see, for example,
\cite{Sto}) that%
\[
\mathfrak{N}_{\boldsymbol{\Omega}}^{i}\approx\pi_{i+\ell}\left(
T(\widehat{\gamma}_{\boldsymbol{\Omega}}^{\ell})\right)  \text{ \ and
\ }\mathfrak{N}_{\boldsymbol{\Omega}_{\mathcal{K}}^{1}}^{i}\approx\pi_{i+\ell
}\left(  T(\widehat{\gamma}_{\boldsymbol{\Omega}_{\mathcal{K}}^{1}}^{\ell
})\right)  .
\]
Consequently, any element of $\pi_{i+\ell}\left(  T(\widehat{\gamma
}_{\boldsymbol{\Omega}}^{\ell})\right)  $ and $\pi_{i+\ell}\left(
T(\widehat{\gamma}_{\boldsymbol{\Omega}_{\mathcal{K}}^{1}}^{\ell})\right)  $
is of order two. This together with (7.1) proves Theorem 1.2.
\end{proof}

By Lemmas 5.2 and 5.4 it will be easy to see inductively that $\Omega
_{\mathcal{K}}(n+m,p+m)$ is the subset of $J^{k}(n+m,p+m)$ which consists of
all $\mathcal{K}$-orbits $\mathcal{K}(i_{+1}(z))$ for $k$-jets $z\in
\Omega_{\mathcal{K}}(n+m-1,p+m-1)$ for $m>1$ and that if $\Omega
(n,p)=\Omega^{I}(n,p)\neq\emptyset$, then we have $\Omega^{I}(n+m,p+m)=\Omega
_{\mathcal{K}}(n+m,p+m)$. By using the map $J^{k}(n+m,p+m)\rightarrow
J^{k}(n+m+q,p+m+q)$ sending a jet $j_{0}^{k}f$ to $j_{0}^{k}(f\times
id_{\mathbb{R}^{q}})$, we have the canonical map $\boldsymbol{\Omega
}_{\mathcal{K}}^{m}\rightarrow\boldsymbol{\Omega}_{\mathcal{K}}^{m+q}$. This
induces the canonical bundle map $\mathbf{b}(\widehat{\gamma})^{(\Omega
_{\mathcal{K}}^{m},\Omega_{\mathcal{K}}^{m+q})}:\widehat{\gamma}%
_{\boldsymbol{\Omega}_{\mathcal{K}}^{m}}^{\ell}\rightarrow\widehat{\gamma
}_{\boldsymbol{\Omega}_{\mathcal{K}}^{m+q}}^{\ell}$ and the associated map
$T(\mathbf{b}(\widehat{\gamma})^{(\Omega_{\mathcal{K}}^{m},\Omega
_{\mathcal{K}}^{m+q})}):T(\widehat{\gamma}_{\boldsymbol{\Omega}_{\mathcal{K}%
}^{m}}^{\ell})\rightarrow T(\widehat{\gamma}_{\boldsymbol{\Omega}%
_{\mathcal{K}}^{m+q}}^{\ell})$. Let $\theta$ denote an integer such that
$\theta=1$ when $n<p$ and $\theta=n+2-p$\ when $n\geqq p\geqq2$.

\begin{proposition}
Let $\theta$ be the integer as above. Then the homomorphism
\[
\lim_{\ell\rightarrow\infty}\pi_{n+\ell}\left(  T(\widehat{\gamma
}_{\boldsymbol{\Omega}_{\mathcal{K}}^{\theta}}^{\ell})\right)  \longrightarrow
\lim_{\ell\rightarrow\infty}\pi_{n+\ell}\left(  T(\widehat{\gamma
}_{\boldsymbol{\Omega}_{\mathcal{K}}^{\theta+q}}^{\ell})\right)
\]
induced from the above map $T(\mathbf{b}(\widehat{\gamma})^{(\Omega
_{\mathcal{K}}^{\theta},\Omega_{\mathcal{K}}^{\theta+q})})$ is an isomorphism
for $q\geqq0$.
\end{proposition}

\begin{proof}
We only prove the unoriented case. The proof proceeds as in the Proof of
Theorem 1.2. By the iterated use of Lemma 7.2, we have that $T(\mathbf{b}%
(\widehat{\gamma})^{(\Omega_{\mathcal{K}}^{\theta},\Omega_{\mathcal{K}%
}^{\theta+q})})$ is a homotopy $\min\{n+\theta,p+\theta-1\}$-equivalence, and
hence, a homotopy $(n+1)$-equivalence. Therefore,
\[
(T(\mathbf{b}(\widehat{\gamma})^{(\Omega_{\mathcal{K}}^{\theta},\Omega
_{\mathcal{K}}^{\theta+q})}))_{\ast}:H_{i}(\boldsymbol{\Omega}_{\mathcal{K}%
}^{\theta})\longrightarrow H_{i}(\boldsymbol{\Omega}_{\mathcal{K}}^{\theta
+q})
\]
is an isomorphism for $0\leqq i\leqq n$ and and an epimorphism for $i=n+1$. By
the argument similar to that in the Proof of Theorem 1.2, we have that
\begin{equation}
(T(\mathbf{b}(\widehat{\gamma})^{(\Omega_{\mathcal{K}}^{\theta},\Omega
_{\mathcal{K}}^{\theta+q})}))_{\ast}:\pi_{i+\ell}\left(  T(\widehat{\gamma
}_{\boldsymbol{\Omega}_{\mathcal{K}}^{\theta}}^{\ell})\right)  \longrightarrow
\pi_{i+\ell}\left(  T(\widehat{\gamma}_{\boldsymbol{\Omega}_{\mathcal{K}%
}^{\theta+q}}^{\ell})\right)
\end{equation}
is a $\mathcal{C}$-isomorphism for $-\ell\leqq i\leqq n$ and a $\mathcal{C}%
$-epimorphism for $i=n+1$. Since any element of the groups in (7.2) is of
order two, we obtain the proposition.
\end{proof}

\section{Classifying space}

In this section we will induce the classifying space $B_{\mathfrak{O}}$\ in
Theorem 1.3.

We consider the vector bundle $J^{k}(\gamma_{G_{n+m}}^{n+m},\mathbb{R}^{p+m})$
with the projection $\pi_{G_{n+m}}^{k}$ onto $G_{n+m}:=G_{n+m}\times\{$a
point$\}$ and the open subbundle $\Omega_{\mathcal{K}}(\gamma_{G_{n+m}}%
^{n+m},\mathbb{R}^{p+m})$ of $J^{k}(\gamma_{G_{n+m}}^{n+m},\mathbb{R}^{p+m})$
associated to $\Omega_{\mathcal{K}}^{m}$, where $\mathbb{R}^{p+m}$ is regarded
as a vector bundle over a point. As in Section 3, the spaces $\{T((\pi_{G_{m}%
}^{k})^{\ast}(\widehat{\gamma}_{G_{m}}^{\ell})|_{\mathbf{\Omega}_{\mathcal{K}%
}})\}_{\ell}$ constitute a spectrum.

Let $\varkappa$\ denote any integer with $\varkappa\geqq p+3$. Let $\ell\gg
n,p,\varkappa$. We set $\mathbb{G}_{n+\theta+\varkappa}=A_{n+\theta
+\varkappa,\ell+p}$. Let $P$\ be embedded in $\mathbb{R}^{p+\varkappa}$. Let
$x\in G_{n+\theta}$, $y\in P$ and $(\nu_{P})_{y}$ be the orthogonal complement
of $T_{y}P$ in $\mathbb{R}^{p+\varkappa}$. Let $c:G_{n+\theta}\times
P\rightarrow\mathbb{G}_{n+\theta+\varkappa}$\ denote the map such that
$c(x,y)$ is the $(n+\theta+\varkappa)$-subspace $x\oplus(\nu_{P})_{y}$ in
$\mathbb{R}^{n+\theta+\varkappa+\ell+p}$. Let $\mathbf{c}:\pi_{G_{n+\theta}%
}^{\ast}(\gamma_{G_{n+\theta}}^{n+\theta})\oplus\pi_{P}^{\ast}(\nu
_{P})\rightarrow\gamma_{\mathbb{G}_{n+\theta+\varkappa}}^{n+\theta+\varkappa}%
$\ be the bundle map which is canonically induced to cover the classifying map
$c$.

Let $\mathbf{J}^{k}(\gamma^{n+\theta}\oplus\nu,TP\oplus\varepsilon_{P}%
^{\theta}\oplus\nu_{P})$ denote%
\[
\mathrm{Hom}\left(
{\displaystyle\bigoplus\limits_{i=1}^{k}}
S^{i}(\pi_{G_{n+\theta}}^{\ast}(\gamma_{G_{n+\theta}}^{n+\theta})\oplus\pi
_{P}^{\ast}(\nu_{P})),\pi_{P}^{\ast}(TP\oplus\varepsilon_{P}^{\theta}\oplus
\nu_{P})\right)
\]
and define the fiber map%
\begin{equation}
J^{k}(\gamma_{G_{n+\theta}}^{n+\theta},TP\oplus\varepsilon_{P}^{\theta
})\longrightarrow\mathbf{J}^{k}(\gamma^{n+\theta}\oplus\nu,TP\oplus
\varepsilon_{P}^{\theta}\oplus\nu_{P})
\end{equation}
over $G_{n+\theta}\times P$ by mapping $j_{x}^{k}\alpha\in J_{x,y}^{k}%
(\gamma_{G_{n+\theta}}^{n+\theta},TP\oplus\varepsilon_{P}^{\theta})$\ to
$j_{(x,y)}^{k}(\alpha\times id_{(\nu_{P})_{y}})$. We also obtain the bundle
map%
\begin{equation}
\mathbf{J}^{k}(\gamma^{n+\theta}\oplus\nu,TP\oplus\varepsilon_{P}^{\theta
}\oplus\nu_{P})\longrightarrow J^{k}(\gamma_{\mathbb{G}_{n+\theta+\varkappa}%
}^{n+\theta+\varkappa},TP\oplus\varepsilon_{P}^{\theta}\oplus\nu_{P})
\end{equation}
covering $(c,\pi_{P}):G_{n+\theta}\times P\rightarrow\mathbb{G}_{n+\theta
+\varkappa}\times P$\ which is canonically induced from the bundle map
$\mathbf{c}$\ and $id_{TP\oplus\varepsilon_{P}^{\theta}\oplus\nu_{P}}$.\ It
follows from Lemma 7.1 that $(c,\pi_{P})$\ is a homotopy $(n+\theta
)$-equivalence. The composite of the maps in (8.1) and (8.2) on the fibers
over $(x,y)$\ and $(c(x,y),y)$ induces a map%
\[
\Omega_{\mathcal{K}}((\gamma_{G_{n+\theta}}^{n+\theta})_{x},T_{y}%
P\oplus\varepsilon_{y}^{\theta})\longrightarrow\Omega_{\mathcal{K}}%
((\gamma_{\mathbb{G}_{n+\theta+\varkappa}}^{n+\theta+\varkappa})_{c(x,y)}%
,T_{y}P\oplus\varepsilon_{y}^{\theta}\oplus(\nu_{P})_{y}),
\]
which is a homotopy $(n+1)$-equivalence by the iterated use of Proposition 6.3
and $G_{n+\theta}\times P\rightarrow\mathbb{G}_{n+\theta+\varkappa}\times P$
is a homotopy $(n+\theta)$-equivalence by Lemma 7.1. Consequently, we obtain
the fiber map%
\begin{equation}
\mathbf{j}_{\Omega_{\mathcal{K}}}:\Omega_{\mathcal{K}}(\gamma_{G_{n+\theta}%
}^{n+\theta},TP\oplus\varepsilon_{P}^{\theta})\longrightarrow\Omega
_{\mathcal{K}}(\gamma_{\mathbb{G}_{n+\theta+\varkappa}}^{n+\theta+\varkappa
},TP\oplus\varepsilon_{P}^{\theta}\oplus\nu_{P})
\end{equation}
covering $(c,\pi_{P})$, which is also a homotopy $(n+1)$-equivalence.

Trivialization $TP\oplus\nu_{P}\rightarrow\varepsilon_{P}^{p+\varkappa}$
induces the bundle isomorphism%
\[
J^{k}(\gamma_{\mathbb{G}_{n+\theta+\varkappa}}^{n+\theta+\varkappa}%
,TP\oplus\varepsilon_{P}^{\theta}\oplus\nu_{P})\longrightarrow J^{k}%
(\gamma_{\mathbb{G}_{n+\theta+\varkappa}}^{n+\theta+\varkappa},\mathbb{R}%
^{p+\theta+\varkappa})\times P
\]
over $\mathbb{G}_{n+\theta+\varkappa}\times P$, where $\mathbb{R}%
^{p+\theta+\varkappa}$\ is regarded as the trivial vector bundle over a
point.\ Let $\Omega_{\mathcal{K}}(\gamma_{\mathbb{G}_{n+\theta+\varkappa}%
}^{n+\theta+\varkappa},\mathbb{R}^{p+\theta+\varkappa})$ be the open subbundle
associated to $\Omega_{\mathcal{K}}^{\theta+\varkappa}$. Then we have the
bundle map%
\begin{equation}
\mathbf{k}_{\Omega_{\mathcal{K}}}:\Omega_{\mathcal{K}}(\gamma_{\mathbb{G}%
_{n+\theta+\varkappa}}^{n+\theta+\varkappa},TP\oplus\varepsilon_{P}^{\theta
}\oplus\nu_{P})\longrightarrow\Omega_{\mathcal{K}}(\gamma_{\mathbb{G}%
_{n+\theta+\varkappa}}^{n+\theta+\varkappa},\mathbb{R}^{p+\theta+\varkappa
})\times P
\end{equation}
over $\mathbb{G}_{n+\theta+\varkappa}\times P$. In the following we identify
the two spaces in (8.4). Thus we have the following lemma.

\begin{lemma}
The fiber map $\mathbf{k}_{\Omega_{\mathcal{K}}}\circ\mathbf{j}_{\Omega
_{\mathcal{K}}}$ covering $(c,\pi_{P})$\ is a homotopy $(n+1)$-equivalence.
\end{lemma}

Let $\pi_{\mathbb{G}_{n+\theta+\varkappa}}^{k}:J^{k}(\gamma_{\mathbb{G}%
_{n+\theta+\varkappa}}^{n+\theta+\varkappa},\mathbb{R}^{p+\theta+\varkappa
})\rightarrow\mathbb{G}_{n+\theta+\varkappa}$ be the canonical projection. Let
$B(\widehat{\gamma}_{\mathbb{G}_{n+\theta+\varkappa}}^{\ell+p},\nu
_{P})_{\Omega_{\mathcal{K}}\times P}$ denote the vector bundle over
$\Omega_{\mathcal{K}}(\gamma_{\mathbb{G}_{n+\theta+\varkappa}}^{n+\theta
+\varkappa},\mathbb{R}^{p+\theta+\varkappa})\times P$ defined by%
\[
B(\widehat{\gamma}_{\mathbb{G}_{n+\theta+\varkappa}}^{\ell+p},\nu_{P}%
)_{\Omega_{\mathcal{K}}\times P}=(\pi_{\mathbb{G}_{n+\theta+\varkappa}}%
^{k})^{\ast}(\widehat{\gamma}_{\mathbb{G}_{n+\theta+\varkappa}}^{\ell
+p})|_{\Omega_{\mathcal{K}}(\gamma_{\mathbb{G}_{n+\theta+\varkappa}}%
^{n+\theta+\varkappa},\mathbb{R}^{p+\theta+\varkappa})}\times\nu_{P}.
\]
This satisfies%
\[
T(B(\widehat{\gamma}_{\mathbb{G}_{n+\theta+\varkappa}}^{\ell+p},\nu
_{P})_{\Omega_{\mathcal{K}}\times P})=T((\pi_{\mathbb{G}_{n+\theta+\varkappa}%
}^{k})^{\ast}(\widehat{\gamma}_{\mathbb{G}_{n+\theta+\varkappa}}^{\ell
+p})|_{\Omega_{\mathcal{K}}(\gamma_{\mathbb{G}_{n+\theta+\varkappa}}%
^{n+\theta+\varkappa},\mathbb{R}^{p+\theta+\varkappa})})\wedge T(\nu_{P}).
\]
Let%
\[%
\begin{array}
[c]{r}%
\mathbf{b}_{\mathbf{k}_{\Omega_{\mathcal{K}}}\circ\mathbf{j}_{\Omega
_{\mathcal{K}}}}:(\pi_{G_{n+\theta}}^{k})^{\ast}(\widehat{\gamma}%
_{G_{n+\theta}}^{\ell})\oplus(\pi_{P}^{k})^{\ast}(TP)|_{\Omega_{\mathcal{K}%
}(\gamma_{\mathbb{G}_{n+\theta}}^{n+\theta},\mathbb{R}^{p+\theta})}\\
\longrightarrow((\pi_{\mathbb{G}_{n+\theta+\varkappa}}^{k})^{\ast}%
(\widehat{\gamma}_{\mathbb{G}_{n+\theta+\varkappa}}^{\ell+p})|_{\Omega
_{\mathcal{K}}(\gamma_{\mathbb{G}_{n+\theta+\varkappa}}^{n+\theta+\varkappa
},\mathbb{R}^{p+\theta+\varkappa})})\times P
\end{array}
\]
be the bundle map covering $\mathbf{k}_{\Omega_{\mathcal{K}}}\circ
\mathbf{j}_{\Omega_{\mathcal{K}}}$ mapping $(\gamma_{G_{n+\theta}}^{n+\theta
})_{x}^{\perp}\oplus T_{y}P$ to $((\gamma_{G_{n+\theta}}^{n+\theta})_{x}%
\oplus(\nu_{P})_{y})^{\perp}$, where $\perp$ denotes the orthogonal
complement. This induces a bundle map%
\[%
\begin{array}
[c]{r}%
\mathbf{B}:((\pi_{G_{n+\theta}}^{k})^{\ast}(\widehat{\gamma}_{G_{n+\theta}%
}^{\ell})\oplus(\pi_{P}^{k})^{\ast}(TP\oplus\nu_{P}))|_{\Omega_{\mathcal{K}%
}(\gamma_{G_{n+\theta}}^{n+\theta},TP\oplus\varepsilon_{P}^{\theta})}\\
\longrightarrow B(\widehat{\gamma}_{\mathbb{G}_{n+\theta+\varkappa}}^{\ell
+p},\nu_{P})_{\Omega_{\mathcal{K}}\times P}%
\end{array}
\]
covering%
\[
\mathbf{k}_{\Omega_{\mathcal{K}}}\circ\mathbf{j}_{\Omega_{\mathcal{K}}}%
:\Omega_{\mathcal{K}}(\gamma_{G_{n+\theta}}^{n+\theta},TP\oplus\varepsilon
_{P}^{\theta})\rightarrow\Omega_{\mathcal{K}}(\gamma_{\mathbb{G}%
_{n+\theta+\varkappa}}^{n+\theta+\varkappa},\mathbb{R}^{p+\theta+\varkappa
})\times P,
\]
which is a bundle map over $P$ and is a homotopy $(n+1)$-equivalence.

Let $X$ and $Y$\ be connected polyhedra with base points respectively.\ Let
$\{X;Y\}$\ denote the set of $S$-homotopy classes of $S$-maps. Let
$G_{n+\theta+\varkappa}=A_{n+\theta+\varkappa,\ell}$ as in Introduction.

\begin{proposition}
Let $\Omega(n,p)$ denote a nonempty $\mathcal{K}$-invariant open subset. Then
there exists an isomorphism%
\[%
\begin{array}
[c]{l}%
\underset{\ell\rightarrow\infty}{\lim}\pi_{n+\ell}\left(  T(\widehat{\gamma
}_{\Omega_{\mathcal{K}}(\gamma_{G_{n+\theta}}^{n+\theta},TP\oplus
\varepsilon_{P}^{\theta})}^{\ell})\right)  \longrightarrow\\
\underset{\ell\rightarrow\infty}{\lim}\pi_{n+\ell+\varkappa}\left(
T((\pi_{G_{n+\theta+\varkappa}}^{k})^{\ast}(\widehat{\gamma}_{G_{n+\theta
+\varkappa}}^{\ell})|_{\Omega_{\mathcal{K}}(\gamma_{G_{n+\theta+\varkappa}%
}^{n+\theta+\varkappa},\mathbb{R}^{p+\theta+\varkappa})})\wedge T(\nu
_{P})\right)  .
\end{array}
\]

\end{proposition}

\begin{proof}
Setting $\boldsymbol{\Omega}_{\mathcal{K}}^{\theta}=\Omega_{\mathcal{K}%
}(\gamma_{G_{n+\theta}}^{n+\theta},TP\oplus\varepsilon_{P}^{\theta})$, we have
that%
\begin{align}
\pi_{n+\ell}\left(  T(\widehat{\gamma}_{\boldsymbol{\Omega}_{\mathcal{K}%
}^{\theta}}^{\ell})\right)   &  \approx\{S^{n+\ell};T(\widehat{\gamma
}_{\boldsymbol{\Omega}_{\mathcal{K}}^{\theta}}^{\ell})\}\\
&  \approx\{S^{n+\ell+p+\varkappa};T(\widehat{\gamma}_{\boldsymbol{\Omega
}_{\mathcal{K}}^{\theta}}^{\ell})\wedge S^{p+\varkappa}\}\nonumber\\
&  \approx\{S^{n+\ell+p+\varkappa};T(\widehat{\gamma}_{\boldsymbol{\Omega
}_{\mathcal{K}}^{\theta}}^{\ell}\oplus\varepsilon_{\boldsymbol{\Omega
}_{\mathcal{K}}^{\theta}}^{p+\varkappa})\}\nonumber\\
&  \approx\{S^{n+\ell+p+\varkappa};T(\widehat{\gamma}_{\boldsymbol{\Omega
}_{\mathcal{K}}^{\theta}}^{\ell}\oplus(\pi_{P}^{k})^{\ast}(TP\oplus\nu
_{P})|_{\boldsymbol{\Omega}_{\mathcal{K}}^{\theta}})\}.\nonumber
\end{align}
By Lemma 8.1 and an argument as in the proof of Theorem 1.2 using the Thom
isomorphism Theorem, the associated map $T(\mathbf{B})$ induces the
isomorphism between the last group in (8.5) and $\{S^{n+\ell+p+\varkappa
};T(B(\widehat{\gamma}_{\mathbb{G}_{n+\theta+\varkappa}}^{\ell+p},\nu
_{P})_{\Omega_{\mathcal{K}}\times P})\}$. Furthermore, we have that
\begin{align}
&  \{S^{n+\ell+p+\varkappa};T(B(\widehat{\gamma}_{\mathbb{G}_{n+\theta
+\varkappa}}^{\ell+p},\nu_{P})_{\Omega_{\mathcal{K}}\times P})\}\\
&  \approx\{S^{n+\ell+p+\varkappa};T((\pi_{\mathbb{G}_{n+\theta+\varkappa}%
}^{k})^{\ast}(\widehat{\gamma}_{\mathbb{G}_{n+\theta+\varkappa}}^{\ell
+p})|_{\Omega_{\mathcal{K}}(\gamma_{\mathbb{G}_{n+\theta+\varkappa}}%
^{n+\theta+\varkappa},\mathbb{R}^{p+\theta+\varkappa})})\wedge T(\nu
_{P})\}\nonumber\\
&  \approx\{S^{n+\ell+\varkappa};T((\pi_{G_{n+\theta+\varkappa}}^{k})^{\ast
}(\widehat{\gamma}_{G_{n+\theta+\varkappa}}^{\ell})|_{\Omega_{\mathcal{K}%
}(\gamma_{G_{n+\theta+\varkappa}}^{n+\theta+\varkappa},\mathbb{R}%
^{p+\theta+\varkappa})})\wedge T(\nu_{P})\}.\nonumber
\end{align}
This proves the proposition.
\end{proof}

In the rest of this section we work in the oriented case and $P$\ should be
connected and oriented. Let $P^{0}$\ be the union of $P$ and the base point.
Consider the duality map $T(\nu_{P})\wedge S^{t}(P^{0})\rightarrow
S^{p+\varkappa+t}$ in \cite{SpaDual}. If $P$ is connected in addition, then
the last group in (8.6) in the oriented case is isomorphic to%
\begin{equation}
\{S^{n+\ell+\varkappa}\wedge S^{t}(P^{0});T((\pi_{\widetilde{G}_{n+\theta
+\varkappa,\ell}}^{k})^{\ast}(\widehat{\gamma}_{\widetilde{G}_{n+\theta
+\varkappa,\ell}}^{\ell})|_{\Omega_{\mathcal{K}}(\gamma_{\widetilde
{G}_{n+\theta+\varkappa,\ell}}^{n+\theta+\varkappa},\mathbb{R}^{p+\theta
+\varkappa})})\wedge S^{p+\varkappa+t}\}.
\end{equation}
Take a representative map $\alpha$ in a homotopy class in (8.7) and consider
the correspondence of a point $y$ in $P$ to a map $\alpha|(S^{n+\ell
+\varkappa+t}\wedge\{y,$base point$\})$. It is not difficult to see that the
set of $S$-homotopy classes in (8.7) is bijective to the following sets of
homotopy classes%
\begin{align*}
&  [P,C^{0}(S^{n+\ell+\varkappa+t},T((\pi_{\widetilde{G}_{n+\theta
+\varkappa,\ell}}^{k})^{\ast}(\widehat{\gamma}_{\widetilde{G}_{n+\theta
+\varkappa,\ell}}^{\ell})|_{\Omega_{\mathcal{K}}(\gamma_{\widetilde
{G}_{n+\theta+\varkappa,\ell}}^{n+\theta+\varkappa},\mathbb{R}^{p+\theta
+\varkappa})})\wedge S^{p+\varkappa+t})]\\
&  \approx\lbrack P,C^{0}(S^{\ell+n-p},T((\pi_{\widetilde{G}_{n+\theta
+\varkappa,\ell}}^{k})^{\ast}(\widehat{\gamma}_{\widetilde{G}_{n+\theta
+\varkappa,\ell}}^{\ell})|_{\Omega_{\mathcal{K}}(\gamma_{\widetilde
{G}_{n+\theta+\varkappa,\ell}}^{n+\theta+\varkappa},\mathbb{R}^{p+\theta
+\varkappa})}))].
\end{align*}
Setting%
\[
B_{\mathfrak{O}}^{x}=\lim_{\ell\rightarrow\infty}C^{0}(S^{\ell+n-p}%
,T((\pi_{\widetilde{G}_{n+\theta+\varkappa,\ell}}^{k})^{\ast}(\widehat{\gamma
}_{\widetilde{G}_{n+\theta+\varkappa,\ell}}^{\ell})|_{\Omega_{\mathcal{K}%
}(\gamma_{\widetilde{G}_{n+\theta+\varkappa,\ell}}^{n+\theta+\varkappa
},\mathbb{R}^{p+\theta+\varkappa})})),
\]
we define the classifying space $B_{\mathfrak{O}}=\lim_{\varkappa
\rightarrow\infty}B_{\mathfrak{O}}^{\varkappa}$ as in Introduction. We have
the following proposition.

\begin{proposition}
Let $\Omega(n,p)$ be as in Proposition 8.2. Let $P$ be a closed connected and
oriented $p$-dimensional manifold. Then there exists an isomorphism%
\[%
\begin{array}
[c]{c}%
\underset{\ell\rightarrow\infty}{\lim}\pi_{n+\ell+\varkappa}\left(
T((\pi_{\widetilde{G}_{n+\theta+\varkappa,\ell}}^{k})^{\ast}(\widehat{\gamma
}_{\widetilde{G}_{n+\theta+\varkappa,\ell}}^{\ell})|_{\Omega_{\mathcal{K}%
}(\gamma_{\widetilde{G}_{n+\theta+\varkappa,\ell}}^{n+\theta+\varkappa
},\mathbb{R}^{p+\theta+\varkappa})})\wedge T(\nu_{P})\right) \\
\longrightarrow\lbrack P,B_{\mathfrak{O}}].
\end{array}
\]

\end{proposition}

\section{Proof of Theorem 1.3}

The development of the h-principles has been described in detail in \cite{G2}.
We only refer to the Smale-Hirsch Immersion Theorem (\cite{Smale},
\cite{HirsImm}), the Feit $k$-mersion Theorem (\cite{Feit}) and the general
theorems due to Gromov\cite{G1} and du Plessis (\cite{duPMCS}, \cite{duPRS},
\cite{duPContIn}). In particular, du Plessis has proposed a nice condition
called \textquotedblleft extensibility\textquotedblright\ under which the
h-principle holds for $\Omega^{I}$-regular maps or smooth maps with only
$\mathcal{K}$-simple singularities. However, this extensibility condition is
not so effective in the dimensions $n\geqq p$. On the other hand,
\`{E}lia\v{s}berg (\cite{E1}, \cite{E2}) has proved the famous h-principle on
the $1$-jet level for sections $s:N\rightarrow\Omega^{1}(N,P)$ which have a
given fold map $f_{0}$ defined around $s^{-1}(\Sigma^{1}(N,P))$ such that
$(j^{2}f_{0})^{-1}(\Sigma^{1,0}(N,P))=s^{-1}(\Sigma^{1}(N,P))$ and the fold
singularities of any semi-index of $f_{0}$ are not empty.

In order to prove Theorem 1.3 by applying Theorems 1.1 and 1.2 we have to show
that the assumption concerning h-principles is satisfied in the situation of
Theorem 1.3. We have proved the h-principle in (h-P) for fold-maps in
\cite{FoldKyoto} and \cite{Exist}. Recently we have introduced a very
effective condition for the h-principle in (h-P) in \cite{K-class}. As an
application we can prove the following theorem by using \cite{K-class} (see a
proof in \cite{Andoarxiv}). Let $k\gg n,p$ as in Introduction.

\begin{theorem}
Let $n<p$ or $n\geqq p\geqq2$. Let $\Omega(n,p)$ denote a $\mathcal{K}%
$-invariant open subspace in $J^{k}(n,p)$ such that when $n\geqq p\geqq2$,
$\Omega(n,p)$ contains $\Sigma^{n-p+1,0}(n,p)$ at least. Then the h-principle
in (h-P) holds for $\Omega(n,p).$
\end{theorem}

In particular, we show the following examples of $\Omega(n,p)$:

(i) $\Omega^{I}(n,p)$ such that when $n\geqq p\geqq2$, $I\geqq(n-p+1,0)$
(\cite{hPrinBo}),

(ii) an open subspace consisting of all regular $k$-jets and a finite number
of $\mathcal{K}$-orbits of $\mathcal{K}$-simple singularities such that when
$n\geqq p\geqq2$, it contains all fold jets in addition (\cite{K-class}).

If $\Omega_{\star}(n+1,p+1)=\Omega_{\mathcal{K}}^{1}$, then we write
$\mathfrak{C}(n,P;\Omega)$ simply in place of $\mathfrak{C}(n,P;\Omega
,\Omega_{\mathcal{K}})$ in the following. We have the following corollary of
Theorems 1.1, 1.2 and 9.1.

\begin{corollary}
Let $n<p$. Let $\Omega(n,p)$ and $\Omega_{\mathcal{K}}^{1}$ be the
$\mathcal{K}$-invariant open subsets given in Theorem 1.2. Let $P$ be a
$p$-dimensional manifold. Then the homomorphism%
\[
\omega:\mathfrak{C}(n,P;\Omega)\longrightarrow\lim_{\ell\rightarrow\infty}%
\pi_{n+\ell}\left(  T(\widehat{\gamma}_{\boldsymbol{\Omega}_{\mathcal{K}}^{1}%
}^{\ell})\right)
\]
is an isomorphism.
\end{corollary}

If we apply Theorem 9.1 to \cite[Theorem 9.2]{Sady2}, then under the same
assumption of Theorem 1.3, there exists an isomorphism of%
\begin{equation}
\mathfrak{C}(n,P;\Omega)\longrightarrow\lim_{\ell\rightarrow\infty}\pi
_{n+\ell}\left(  T(\widehat{\gamma}_{\boldsymbol{\Omega}_{\mathcal{K}}^{m}%
}^{\ell})\right)
\end{equation}
for a sufficiently large integer $m$ also in the dimensions $n\geqq p\geqq2$
as well as $n<p$. This isomorphism is nothing but the composite of $\omega$
and the map%
\[
\lim_{\ell\rightarrow\infty}\pi_{n+\ell}\left(  T(\widehat{\gamma
}_{\boldsymbol{\Omega}_{\mathcal{K}}^{1}}^{\ell})\right)  \longrightarrow
\lim_{\ell\rightarrow\infty}\pi_{n+\ell}\left(  T(\widehat{\gamma
}_{\boldsymbol{\Omega}_{\mathcal{K}}^{m}}^{\ell})\right)
\]
which is induced by $T(\mathbf{b}(\widehat{\gamma})^{(\Omega_{\mathcal{K}}%
^{1},\Omega_{\mathcal{K}}^{m})})$.

We are now ready to prove Theorem 1.3.

\begin{proof}
[Proof of Theorem 1.3]If $n<p$, then we have%
\[
\mathfrak{O}(n,P;\Omega)\approx\lim_{\ell\rightarrow\infty}\pi_{n+\ell}\left(
T(\widehat{\gamma}_{\boldsymbol{\Omega}_{\mathcal{K}}^{1}}^{\ell})\right)
\]
by Corollary 9.2. If $n\geqq p\geqq2$, then the homomorphisms%
\begin{equation}
(T(\mathbf{b}(\widehat{\gamma})^{(\Omega_{\mathcal{K}}^{\theta+\varkappa
},\Omega_{\mathcal{K}}^{\theta+\varkappa+q})}))_{\ast}:\pi_{n+\ell}\left(
T(\widehat{\gamma}_{\boldsymbol{\Omega}_{\mathcal{K}}^{\theta+\varkappa}%
}^{\ell})\right)  \longrightarrow\pi_{n+\ell}\left(  T(\widehat{\gamma
}_{\boldsymbol{\Omega}_{\mathcal{K}}^{\theta+\varkappa+q}}^{\ell})\right)
\end{equation}
are isomorphisms for integers $q\geqq0$ by Proposition 7.3. By Propositions
8.2, 8.3 and (9.1) $\mathfrak{O}(n,P;\Omega)$ is isomorphic to
$[P,B_{\mathfrak{O}}]$. This completes the proof.
\end{proof}

\begin{corollary}
Under the same assumption of Theorem 1.3, $\mathfrak{O}(n,P;\Omega
)\otimes\mathbb{Q}$ is isomorphic to $H_{n}(\Omega_{\mathcal{K}}%
(\gamma_{G_{n+\theta+\varkappa}}^{n+\theta+\varkappa},\mathbb{R}%
^{p+\theta+\varkappa})\times P;\mathbb{Q)}$ in the dimensions $n<p$ and
$n\geqq p\geqq2$.
\end{corollary}

Here, let us see a relationship between the space $B_{\mathfrak{O}}$, the
Thom-Atiyah duality in bordism and cobordism and the spaces introduced by
Kazarian(\cite{Kaz1} and \cite{Kaz2}). Let us recall the $n$-dimensional
oriented bordism group\ $\Omega_{n}(P)$ of maps to $P$\ and the Thom-Atiyah
duality, $\pi_{n+\ell}(MSO(\ell)\wedge P^{0})\approx\lbrack P,\Omega^{\ell
}MSO(\ell+p-n)]$\ (see \cite{Atiyah} and \cite[Chap. I, 12 ]{CoFl}), where
$MSO$ denotes the Thom space of the universal bundle $\gamma$ over $BSO$ and
$\Omega^{\ell}X$\ denotes the $\ell$-th iterated loop space. We have the
following commutative diagram under the notation in (8.5) ($\ell\gg n,p$),
although we do not give a proof of the commutativity:%
\begin{equation}%
\begin{array}
[c]{ccc}%
\mathfrak{O}(n,P;\Omega) & \longrightarrow & \Omega_{n}(P)\\
\approx\downarrow &  & \\
\mathrm{\operatorname{Im}}\mathfrak{^{\mathfrak{O}}}\left(  T(\mathbf{b}%
(\widehat{\gamma})^{(\Omega,\Omega_{\mathcal{K}})})\right)  (\subset
\pi_{n+\ell}\left(  T(\widehat{\gamma}_{\boldsymbol{\Omega}_{\mathcal{K}}^{1}%
}^{\ell})\right)  ) &  & \downarrow\approx\\
\approx\downarrow &  & \\
\pi_{n+\ell+p+\varkappa}(T(\widehat{\gamma}_{\boldsymbol{\Omega}_{\mathcal{K}%
}^{2}}^{\ell})\wedge S^{p+\varkappa}(P^{0})) & \longrightarrow & \pi
_{n+\ell+p+\varkappa}(MSO(\ell)\wedge S^{p+\varkappa}(P^{0}))\\
\approx\downarrow &  & \downarrow\approx\\
\lbrack P,B_{\mathfrak{O}}] & \longrightarrow & [P,\Omega^{\ell+n-p}%
MSO(\ell)],
\end{array}
\end{equation}
where

(1) $\boldsymbol{\Omega}_{\mathcal{K}}^{1}=\Omega_{\mathcal{K}}(\gamma
_{\widetilde{G}_{n+1,\ell}}^{n+1},TP\oplus\varepsilon_{P}^{1})$ and
$\boldsymbol{\Omega}_{\mathcal{K}}^{2}=\Omega_{\mathcal{K}}(\gamma
_{\widetilde{G}_{n+\theta+\varkappa,\ell}}^{n+\theta+\varkappa},\mathbb{R}%
^{p+\theta+\varkappa}),$

(2) the second and the bottom horizontal homomorphisms are induced from the
canonical projections%
\[
\Omega_{\mathcal{K}}(\gamma_{\widetilde{G}_{n+\theta+\varkappa,\ell}%
}^{n+\theta+\varkappa},\mathbb{R}^{p+\theta+\varkappa})\subset J^{k}%
(\gamma_{\widetilde{G}_{n+\theta+\varkappa,\ell}}^{n+\theta+\varkappa
},\mathbb{R}^{p+\theta+\varkappa})\longrightarrow\widetilde{G}_{n+\theta
+\varkappa,\ell}\subset BSO(\ell),
\]

(3) the left vertical map is an isomorphism by Theorem 1.2, (9.1),
Propositions 8.2 and 8.3 for $n<p$ and $n\geqq p\geqq2$.

According to \cite{Kaz1} and \cite[2.18 Corollary and 2.8 Example]{Kaz2}, let
us consider the subspace $\Omega^{\ell+n-p}MSO(\ell)_{\Omega(n,p)}$ in
$\Omega^{\ell+n-p}MSO(\ell)$ which consists of all maps $a:S^{\ell
+n-p}\rightarrow MSO(\ell)$ such that $a$ is smooth around $a^{-1}(BSO(\ell))$
and that for any point $x\in a^{-1}(BSO(\ell))$ with $a(x)=y$, the $k$-jet of
the composite of $a:(S^{\ell+n-p},x)\rightarrow(MSO(\ell),y)$ and a projection
germ of $(MSO(\ell),y)$ to the fiber $(\gamma_{y},y)$ lies in $\Omega
_{\mathcal{K}}(S^{\ell+n-p},\gamma_{y})$ associated to $\Omega(n,p)$. Then the
cobordism class represented by an $\Omega$-regular map is mapped to the
homotopy class represented by a map $P\rightarrow\Omega^{\ell+n-p}%
MSO(\ell)_{\Omega(n,p)}$ in (9.3), whose corresponding map $S^{\ell+n-p}\times
P\rightarrow MSO(\ell)$ is transverse to $BSO(\ell)$.

\section{Fold-maps}

Let $m\geqq q$. Let $V_{m,q}^{\operatorname{row}}$ denote the Stiefel manifold
$(E_{q}\times O(m-q))\diagdown O(m)$ under the canonical bases of
$\mathbb{R}^{m}$ and $\mathbb{R}^{q}$, whose element is regarded as an
epimorphism $\mathbb{R}^{m}\rightarrow\mathbb{R}^{q}$ or a regular $q\times
m$-matrix in the following. Let $\mathcal{E}\rightarrow X$ and $\mathcal{F}%
\rightarrow Y$ be vector bundles of dimensions $m$ and $q$ with metrics
respectively. Let $V(\mathcal{E},\mathcal{F})$ denote the subbundle of
Hom$(\mathcal{E},\mathcal{F})$ associated to $V_{m,q}^{\operatorname{row}}$.
Let $\pi_{X}^{V}:V(\mathcal{E},\mathcal{F})\rightarrow X$ and $\pi_{Y}%
^{V}:V(\mathcal{E},\mathcal{F})\rightarrow Y$ be the canonical projections respectively.

We have the actions of $O(q)\times O(m)$ on $V_{m+1,q}^{\operatorname{row}}$
from the left-hand side through $O(q)$ and from the right-hand side through
$O(m)\times1$ respectively. The group $O(q)\times O(m)$ also naturally acts on
$\Omega^{m-q+1,0}(m,q)$.\ In \cite[Theorem 2.6]{FoldRIMS} we have described
the homotopy type of $\Omega^{m-q+1,0}(m,q)$ in terms of orthogonal groups and
Stiefel manifolds, and have given a topological embedding
\[
i_{V,\Omega}=i_{V,\Omega}^{m,q}:V_{m+1,q}^{\operatorname{row}}%
\mathcal{\rightarrow}\Omega^{m-q+1,0}(m,q),
\]
which is equivariant with respect to the actions of $O(q)\times O(m)$.
Furthermore, if $m-q+1$ is odd, then there exists an equivariant map
\[
R_{\Omega,V}=R_{\Omega,V}^{m,q}:\Omega^{m-q+1,0}(m,q)\mathcal{\rightarrow
}V_{m+1,q}^{\operatorname{row}}%
\]
such that $R_{\Omega,V}\circ i_{V,\Omega}$ is the identity of $V_{m+1,q}%
^{\operatorname{row}}$. In particular, we note that if $m=q$, then
$i_{V,\Omega}\circ R_{\Omega,V}$ is a deformation retraction of $\Omega
^{m-q+1,0}(m,q)$.

If $i_{+1}(j_{0}^{2}f)$ were defined by $i_{+1}(j_{0}^{2}f)=j_{0}%
^{2}(id_{\mathbb{R}}\times f)$, then the following technical modification of
$i_{V,\Omega}$ and $R_{\Omega,V}$\ followed by Lemma 10.1 is unnecessary.\ Let
$h_{t}:\mathbb{R}^{t}\rightarrow\mathbb{R}^{t}$ be the map reversing the order
of coordinates as $h_{t}(x_{1},x_{2},\cdots,x_{t-1},x_{t})=(x_{t}%
,x_{t-1},\cdots,x_{2},x_{t-1})$. Define $\mathfrak{i}_{V,\Omega}%
=\mathfrak{i}_{V,\Omega}^{m,q}:V_{m+1,q}^{\operatorname{row}}%
\mathcal{\rightarrow}\Omega^{m-q+1,0}(m,q)$ and $\mathfrak{R}_{\Omega
,V}=\mathfrak{R}_{\Omega,V}^{m,q}:\Omega^{m-q+1,0}(m,q)\mathcal{\rightarrow
}V_{m+1,q}^{\operatorname{row}}$ by%
\[
\mathfrak{i}_{V,\Omega}^{m,q}(A)=i_{V,\Omega}^{m,q}(Ah_{m+1})\cdot h_{m}\text{
\ \ and \ \ }\mathfrak{R}_{\Omega,V}^{m,q}(j_{0}^{2}f)=R_{\Omega,V}%
^{m,q}(j_{0}^{2}(f\circ h_{m}))h_{m+1}.
\]
It will be easy to see $\mathfrak{R}_{\Omega,V}^{m,q}\circ\mathfrak{i}%
_{V,\Omega}^{m,q}=id_{V_{n+1,p}^{\operatorname{row}}}$. Let $i^{+1}%
:J^{k}(n,p)\rightarrow J^{k}(n+1,p+1)$ denote the map defined by $i^{+1}%
(j_{0}^{k}f)=j_{0}^{k}(id_{\mathbb{R}}\times f)$. Let $j^{+1},$ $j_{+1}%
:V_{n+1,p}^{\operatorname{row}}\rightarrow V_{n+2,p+1}^{\operatorname{row}}$
denote the map defined by $j^{+1}(A)=(1)\dotplus A$ and $j_{+1}(A)=A\dotplus
(1)$, where $\dotplus$ denotes the direct sum of matrices. We consider the
action of $O(p)\times O(n)$ on $V_{n+1,p}^{\operatorname{row}}$\ by
$(S,T)\cdot A=SA((1)\dotplus T^{-1})$ for $S\in O(p)$ and $T\in O(n)$.

\begin{lemma}
Under the above notation, we have that $\mathfrak{i}_{V,\Omega}^{n,p}$ and
$\mathfrak{R}_{\Omega,V}^{n,p}$\ are equivariant with respect to the actions
of $O(p)\times O(n)$ and that
\[
\mathfrak{R}_{\Omega,V}^{n+1,p+1}(i_{+1}(j_{0}^{2}(f)))=j_{+1}(\mathfrak{R}%
_{\Omega,V}^{n,p}(j_{0}^{2}f))\text{ and \ }\mathfrak{i}_{V,\Omega}%
^{n+1,p+1}(j_{+1}(A))=i_{+1}(\mathfrak{i}_{V,\Omega}^{n,p}(A)).
\]

\end{lemma}

\begin{proof}
For $A\in V_{n+1,p}^{\operatorname{row}}$, $S\in O(p)$ and $T\in O(n)$, we
have that%
\begin{align*}
\mathfrak{i}_{V,\Omega}^{n,p}((S,T^{-1})\cdot A)  &  =\mathfrak{i}_{V,\Omega
}^{n,p}(SA((1)\dotplus T)))\\
&  =i_{V,\Omega}^{n,p}(SA((1)\dotplus T)h_{n+1})\cdot h_{n}\\
&  =S\cdot i_{V,\Omega}^{n,p}(Ah_{n+1}h_{n+1}((1)\dotplus T)h_{n+1})\cdot
h_{n}\\
&  =S\cdot i_{V,\Omega}^{n,p}(Ah_{n+1}(h_{n}Th_{n}\dotplus(1)))\cdot h_{n}\\
&  =S\cdot i_{V,\Omega}^{n,p}(Ah_{n+1})\cdot(h_{n}Th_{n}h_{n})\\
&  =S\cdot(i_{V,\Omega}^{n,p}(Ah_{n+1})\cdot h_{n})\cdot T\\
&  =(S,T^{-1})\cdot\mathfrak{i}_{V,\Omega}^{n,p}(A),
\end{align*}%
\begin{align*}
\mathfrak{R}_{\Omega,V}^{n,p}(S,T^{-1})\cdot(j_{0}^{2}f))  &  =\mathfrak{R}%
_{\Omega,V}^{n,p}(S\cdot(j_{0}^{2}f)\cdot T)\\
&  =R_{\Omega,V}^{n,p}(S\cdot(j_{0}^{2}(f)\cdot T)\cdot h_{n}))h_{n+1}\\
&  =SR_{\Omega,V}^{n,p}(j_{0}^{2}(f)\cdot(h_{n}h_{n})T\cdot h_{n})h_{n+1}\\
&  =SR_{\Omega,V}^{n,p}(j_{0}^{2}(f)\cdot h_{n}\cdot(h_{n}Th_{n}))h_{n+1}\\
&  =SR_{\Omega,V}^{n,p}(j_{0}^{2}(f\circ h_{n}))(h_{n}Th_{n}\dotplus
(1))h_{n+1}\\
&  =SR_{\Omega,V}^{n,p}(j_{0}^{2}(f\circ h_{n}))h_{n+1}h_{n+1}(h_{n}%
Th_{n}\dotplus(1))h_{n+1}\\
&  =SR_{\Omega,V}^{n,p}(j_{0}^{2}(f\circ h_{n}))h_{n+1}((1)\dotplus T)\\
&  =(S,T^{-1})\cdot\mathfrak{R}_{\Omega,V}^{n,p}(j_{0}^{2}f).
\end{align*}

By the definition of $i_{V,\Omega}$ and $R_{\Omega,V}$ we have that%
\[
R_{\Omega,V}^{n+1,p+1}(j_{0}^{2}(id_{\mathbb{R}}\times f))=j^{+1}(R_{\Omega
,V}^{n,p}(j_{0}^{2}f))\text{ and }i_{V,\Omega}^{n+1,p+1}(j^{+1}(A))=i^{+1}%
(i_{V,\Omega}^{n,p}(A)).
\]
Therefore, we have that%
\begin{align*}
\mathfrak{R}_{\Omega,V}^{n+1,p+1}(i_{+1}(j_{0}^{2}(f))  &  =R_{\Omega
,V}^{n+1,p+1}(h_{p+1}h_{p+1}\cdot j_{0}^{2}((f\times id_{\mathbb{R}})\cdot
h_{n+1}))h_{n+2}\\
&  =R_{\Omega,V}^{n+1,p+1}(h_{p+1}\cdot j_{0}^{2}(id_{\mathbb{R}}\times
(h_{p}\circ f\circ h_{n})))h_{n+2}\\
&  =h_{p+1}j^{+1}R_{\Omega,V}^{n,p}(j_{0}^{2}(h_{p}\circ f\circ h_{n}%
))h_{n+2}\\
&  =(h_{p}R_{\Omega,V}^{n,p}(j_{0}^{2}(h_{p}\circ f\circ h_{n}))h_{n+1}%
)\dotplus(1)\\
&  =(R_{\Omega,V}^{n,p}(j_{0}^{2}(f\circ h_{n}))h_{n+1})\dotplus(1)\\
&  =j_{+1}(\mathfrak{R}_{\Omega,V}^{n,p}(j_{0}^{2}(f))
\end{align*}
and other formula is similarly proved.
\end{proof}

If we provide $\gamma_{G_{m}}^{m}$ and $\mathcal{F}$ with metrics, then the
structure group of $J^{2}(\gamma_{G_{m}}^{m}\oplus\varepsilon_{G_{m}}%
^{1},\mathcal{F})$ is reduced to $O(q)\times O(m)$. Let $\boldsymbol{\Omega
}=\Omega^{m-q+1,0}(\gamma_{G_{m}}^{m},\mathcal{F})$. Let $\mathfrak{i}%
_{V\mathbf{,}\Omega}(\gamma_{G_{m}}^{m},\mathcal{F}):V(\gamma_{G_{m}}%
^{m}\oplus\varepsilon_{G_{m}}^{1},\mathcal{F})\rightarrow\boldsymbol{\Omega}$
and $\mathfrak{R}_{\Omega,V}(\gamma_{G_{m}}^{m},\mathcal{F}%
):\boldsymbol{\Omega}\rightarrow V(\gamma_{G_{m}}^{m}\oplus\varepsilon_{G_{m}%
}^{1},\mathcal{F})$ be the fiber maps associated to $\mathfrak{i}_{V,\Omega}$
and $\mathfrak{R}_{\Omega,V}$ respectively. Then

$(10$-$\mathrm{i})$ the fiber map $\mathfrak{i}_{V\mathbf{,}\Omega}%
(\gamma_{G_{m}}^{m},\mathcal{F})$ is a topological embedding,

$(10$-$\mathrm{ii})$ if $m-q+1$ is odd, then the composition $\mathfrak{R}%
_{\Omega\mathbf{,}V}(\gamma_{G_{m}}^{m},\mathcal{F})\circ\mathfrak{i}%
_{V\mathbf{,}\Omega}(\gamma_{G_{m}}^{m},\mathcal{F})$ is the identity of
$V(\gamma_{G_{m}}^{m}\oplus\varepsilon_{G_{m}}^{1},\mathcal{F})$.

\noindent Let $\widehat{\gamma}_{V}^{\ell}$ denote the vector bundle induced
from $\widehat{\gamma}_{G_{m}}^{\ell}$ over $V(\gamma_{G_{m}}^{m}%
\oplus\varepsilon_{G_{m}}^{1},\mathcal{F})$. We note $\widehat{\gamma
}_{\mathbf{\Omega}}^{\ell}=(\mathfrak{R}_{\Omega,V}(\gamma_{G_{m}}%
^{m},\mathcal{F}))^{\ast}\widehat{\gamma}_{V}^{\ell}$. Then we have the bundle
maps $\mathbf{b}_{\mathfrak{R}}:\widehat{\gamma}_{\boldsymbol{\Omega}}^{\ell
}\rightarrow\widehat{\gamma}_{V}^{\ell}$ and $\mathbf{b}_{\mathfrak{i}%
}:\widehat{\gamma}_{V}^{\ell}\rightarrow\widehat{\gamma}_{\boldsymbol{\Omega}%
}^{\ell}$ covering $\mathfrak{R}_{\Omega\mathbf{,}V}(\gamma_{G_{m}}%
^{m},\mathcal{F})$ and $\mathfrak{i}_{V\mathbf{,}\Omega}(\gamma_{G_{m}}%
^{m},\mathcal{F})$ respectively. Their associated maps $T(\mathbf{b}%
_{\mathfrak{R}})$\ and $T(\mathbf{b}_{\mathfrak{i}})$ between the Thom spaces
satisfies that $T(\mathbf{b}_{\mathfrak{R}})\circ T(\mathbf{b}_{\mathfrak{i}%
})$ is equal to the identity of $T(\widehat{\gamma}_{V}^{\ell})$.

Let $v^{\Delta}:V_{n+1,p}^{\operatorname{row}}\rightarrow V_{n+2,p+1}%
^{\operatorname{row}}$ denote the map sending an epimorphism $A\in
V_{n+1,p}^{\operatorname{row}}$\ to the direct sum $A\dotplus(1)$ of matrices
$A$ and $(1)$. Set%
\[%
\begin{array}
[c]{cc}%
V_{n+1}=V(\gamma_{G_{n}}^{n}\oplus\varepsilon_{G_{n}}^{1},TP), &
V_{n+2}=V(\gamma_{G_{n+1}}^{n+1}\oplus\varepsilon_{G_{n+1}}^{1},TP\oplus
\varepsilon_{P}^{1}),\\
\boldsymbol{\Omega}=\Omega^{n-p+1,0}(\gamma_{G_{n}}^{n},TP), &
\boldsymbol{\Omega}_{\mathcal{K}}=\Omega^{n-p+1,0}(\gamma_{G_{n+1}}%
^{n+1},TP\oplus\varepsilon_{P}^{1}).
\end{array}
\]
Let $i^{G}$ denote the canonical classifying map $G_{n}\rightarrow G_{n+1}$ of
$\gamma_{G_{n}}^{n}\oplus\varepsilon_{G_{n}}^{1}$.\ Let $V^{\Delta}%
:V_{n+1}\rightarrow V_{n+2}$\ denote the fiber map associated to $v^{\Delta}$
covering $(i^{G},\pi_{P})$. Let $\Omega=\Omega^{n-p+1,0}$. If $n-p+1$ is odd,
then we have the commutative diagram%
\[
\begin{CD}
\Omega^{n-p+1,0}(\gamma_{G_{n}}^{n},TP) @>\Delta^{(\Omega,\Omega_{\mathcal{K}})}>>
\Omega^{n-p+1,0}(\gamma_{G_{n+1}}^{n+1},TP\oplus\varepsilon_{P}^{1})\\
@V\mathfrak{R}_{\Omega,V}^{n,p}(\gamma_{{G}_{n}}^{{n}},TP)VV
@VV\mathfrak{R}_{\Omega,V}^{n+1,p+1}(\gamma_{G_{n+1}}^{n+1},TP\oplus\varepsilon_{P}^{1})V\\
V_{n+1}@>{V^{\Delta}}>>V_{n+2}\\
\end{CD}
\]
by Lemma 10.1, where the horizontal fiber maps cover $(i^{G},\pi_{P})$.

\begin{lemma}
Let $n-p+1$\ be odd. The fiber map $V^{\Delta}$ is a homotopy $n$-equivalence.

\begin{proof}
Since $v^{\Delta}$ and $(i^{G},\pi_{P})$\ are homotopy $n$-equivalences, it
follows that $V^{\Delta}$ is a homotopy $n$-equivalence.
\end{proof}
\end{lemma}

We have the bundle map $\mathbf{b}_{V}:\widehat{\gamma}_{V_{n+1}}^{\ell
}\rightarrow\widehat{\gamma}_{V_{n+2}}^{\ell}$ covering $V^{\Delta}$\ and its
associated map $T(\mathbf{b}_{\mathbf{V}})$ between the Thom spaces. Recalling
the map $T(\mathbf{b}(\widehat{\gamma})^{(\Omega,\Omega_{\mathcal{K}}%
)}):T(\widehat{\gamma}_{\boldsymbol{\Omega}}^{\ell})\longrightarrow
T(\widehat{\gamma}_{\boldsymbol{\Omega}_{\mathcal{K}}}^{\ell})$ we obtain the
following commutative diagram%
\begin{equation}
\begin{CD} \pi_{n+\ell}\left( T(\widehat{\gamma}_{\boldsymbol{\Omega}}^{\ell})\right) @> T(\mathbf{b}({\widehat{\gamma}})^{(\Omega,\Omega_\mathcal{K})})_{*} >> \pi_{n+\ell}\left( T(\widehat{\gamma}_{\boldsymbol{\Omega}_{\mathcal{K}}}^{\ell})\right) \\ @VVV @VVV\\ \pi_{n+\ell}\left( T(\widehat{\gamma}_{V_{n+1}}^{\ell})\right) @> T(\mathbf{b}_{V})_{\ast}>> \pi_{n+\ell}\left( T(\widehat{\gamma}_{V_{n+2}}^{\ell})\right), \end{CD}
\end{equation}
where the left vertical map and the right vertical maps are $T(\mathbf{b}%
_{\mathfrak{R}})_{\ast}$ for $(n,p)$ and $(n+1,p+1)$ respectively. Let
$G_{n+2+\varkappa}=A_{n+2+\varkappa,\ell}$ as above.\ 

\begin{proposition}
Let $n\geqq p\geqq2$ and $\ell\gg n,p$. Let $n-p+1$ be odd. Let
$\boldsymbol{\Omega}=\Omega^{n-p+1,0}(\gamma_{G_{n}}^{n},TP)$ and
$\boldsymbol{\Omega}_{\mathcal{K}}=\Omega^{n-p+1,0}(\gamma_{G_{n+1}}%
^{n+1},TP\oplus\varepsilon_{P}^{1})$ be as above. Then there exist
homomorphisms%
\begin{align*}
\omega_{V}  &  :\mathfrak{C}(n,P;\Omega^{n-p+1,0})\longrightarrow\pi_{n+\ell
}\left(  T(\widehat{\gamma}_{V_{n+2}}^{\ell})\right)  ,\\
\iota_{V}  &  :\pi_{n+\ell}\left(  T(\widehat{\gamma}_{V_{n+2}}^{\ell
})\right)  \longrightarrow\mathfrak{C}(n,P;\Omega^{n-p+1,0})
\end{align*}
such that $\omega_{V}\circ\iota_{V}$ is the identity of $\pi_{n+\ell}\left(
T(\widehat{\gamma}_{V_{n+2}}^{\ell})\right)  $. In particular, if $n=p$, then
$\omega_{V}$ is an isomorphism.
\end{proposition}

\begin{proof}
(i) In the diagram (10.1) it follows from the similar argument as in Proof of
Theorem 1.2 that $T(\mathbf{b}_{V})_{\ast}$ is an epimorphism. It follows from
$(10$-$\mathrm{i})$ and $(10$-$\mathrm{ii})$ that $T(\mathbf{b}_{\mathfrak{R}%
})_{\ast}$ in the diagram (10.1) for $(n,p)$ and $(n+1,p+1)$ are epimorphisms.
By Theorem 1.1 we obtain the required homomorphisms $\omega_{V}$ and
$\iota_{V}$.
\end{proof}

Here, we prove the following refinement of \cite[Theorem 0.3]{FoldRIMS}. This
theorem should be compared with the results in \cite{KalmarNEW} and
\cite{KalmarNEW2}.

\begin{theorem}
Under the assumption of Proposition 10.3 there exists a splitting epimorphism%
\[
\mathfrak{C}(n,P;\Omega^{n-p+1,0})\rightarrow\underset{\varkappa
\rightarrow\infty}{\lim}\left(  \underset{\ell\rightarrow\infty}{\lim}%
\pi_{n+\ell+\varkappa}\left(  T(\widehat{\gamma}_{G_{n-p+1}}^{\ell})\wedge
T(\nu_{P})\right)  \right)  .
\]
In particular, if $n=p$, then this is an isomorphism.
\end{theorem}

\begin{proof}
The proof proceeds similarly as in the proof of Propositions 8.2 and 8.3 by
replacing spaces $\Omega_{\mathcal{K}}(\xi,\mathcal{F})$ by $V(\xi
,\mathcal{F})$. We only prove the unoriented case. Define the fiber map%
\begin{equation}
V_{n+2}\longrightarrow V^{\prime}(\gamma_{G_{n+1}}^{n+1}\oplus\varepsilon
_{G_{n+1}}^{1}\oplus\nu_{P},TP\oplus\varepsilon_{P}^{1}\oplus\nu_{P}),
\end{equation}
where the last space is the subbundle of%
\[
\mathrm{Hom}(\pi_{G_{n+1}}^{\ast}(\gamma_{G_{n+1}}^{n+1}\oplus\varepsilon
_{G_{n+1}}^{1})\oplus\pi_{P}^{\ast}(\nu_{P}),TP\oplus\varepsilon_{P}^{1}%
\oplus\nu_{P})
\]
associated to $V_{n+2+\varkappa,p+1+\varkappa}$ over $G_{n+1}\times P$ and the
map sends $A\in(V_{n+2})_{(x,y)}$\ to $A\oplus id_{(\nu_{P})_{y}}$. We have a
bundle map%
\begin{equation}
\pi_{G_{n+1}}^{\ast}(\gamma_{G_{n+1}}^{n+1}\oplus\varepsilon_{G_{n+1}}%
^{1})\oplus\pi_{P}^{\ast}(\nu_{P})\longrightarrow\gamma_{G_{n+2+\varkappa}%
}^{n+2+\varkappa}%
\end{equation}
covering a classifying map $c^{+}:G_{n+1}\times P\rightarrow G_{n+2+\varkappa
}$ (we may need to replace $\ell$ with a bigger integer). We note that
$(c^{+},\pi_{P}):G_{n+1}\times P\rightarrow G_{n+2+\varkappa}\times P$ is a
homotopy $(n+1)$-equivalence. By (10.2) and (10.3) we obtain a homotopy
$(n+1)$-equivalent fiber map
\begin{equation}
V_{n+2}\longrightarrow V(\gamma_{G_{n+2+\varkappa}}^{n+2+\varkappa}%
,TP\oplus\varepsilon_{P}^{1}\oplus\nu_{P})
\end{equation}
covering $(c^{+},\pi_{P}):G_{n+1}\times P\rightarrow G_{n+2+\varkappa}\times
P$. The trivialization $TP\oplus\nu_{P}\rightarrow\varepsilon_{P}%
^{p+\varkappa}$ induces the bundle map%
\begin{equation}
V(\gamma_{G_{n+2+\varkappa}}^{n+2+\varkappa},TP\oplus\varepsilon_{P}^{1}%
\oplus\nu_{P})\longrightarrow V(\gamma_{G_{n+2+\varkappa}}^{n+2+\varkappa
},\mathbb{R}^{p+1+\varkappa})\times P
\end{equation}
over $G_{n+2+\varkappa}\times P$.

We denote, by $V_{n-p+1,p+1+\varkappa,\ell}$, the space which consists of all
triples $(a,b,c)$\ where $a$, $b$ and $c$\ are mutually perpendicular
subspaces in\ $\mathbb{R}^{n+\ell+\varkappa+2}$ of dimensions $n-p+1$,
$p+1+\varkappa$ and $\ell$ respectively with $a\oplus b\oplus c=\mathbb{R}%
^{n+\ell+\varkappa+2}$. Let $\gamma_{V}^{p+1+\varkappa}$\ be the canonical
vector bundle over $V_{n-p+1,p+1+\varkappa,\ell}$ of dimension $p+1+\varkappa$.

We denote an element of $V(\gamma_{G_{n+2+\varkappa}}^{n+2+\varkappa
},\mathbb{R}^{p+1+\varkappa})$ by $(\alpha,h)$, where $\alpha\in
G_{n+2+\varkappa}$ and $h\in V((\gamma_{G_{n+2+\varkappa}}^{n+2+\varkappa
})_{\alpha},\mathbb{R}^{p+1+\varkappa})$, which is regarded as an epimorphism.
Then $(\alpha,h)$\ defines $(\mathrm{Ker}(h),\mathrm{Ker}(h)^{\perp}%
,\alpha^{\perp})$ in $V_{n-p+1,p+1+\varkappa,\ell}$, which is the triple of
the kernel of $h$, the orthogonal complement of $\mathrm{Ker}(h)$ in $\alpha
$\ and the orthogonal complement $\alpha^{\perp}$. Let $V(\gamma
_{V}^{p+1+\varkappa},\mathbb{R}^{p+1+\varkappa})$\ denote the principal bundle
with fiber $O(p+1+\varkappa)$ associated to $\mathrm{Hom}(\gamma
_{V}^{p+1+\varkappa},\mathbb{R}^{p+1+\varkappa})$. We have the canonical
homeomorphism%
\[
V(\gamma_{G_{n+2+\varkappa}}^{n+2+\varkappa},\mathbb{R}^{p+1+\varkappa
})\rightarrow V(\gamma_{V}^{p+1+\varkappa},\mathbb{R}^{p+1+\varkappa})
\]
which maps $(\alpha,h)$ to $h|\mathrm{Ker}(h)^{\perp}:\mathrm{Ker}(h)^{\perp
}\rightarrow\mathbb{R}^{p+1+\varkappa}$ over $(\mathrm{Ker}(h),\mathrm{Ker}%
(h)^{\perp},\alpha^{\perp})$. Let $\mu$ denote the map%
\[
\mu:V(\gamma_{G_{n+2+\varkappa}}^{n+2+\varkappa},\mathbb{R}^{p+1+\varkappa
})=V(\gamma_{V}^{p+1+\varkappa},\mathbb{R}^{p+1+\varkappa})\longrightarrow
V_{n-p+1,p+1+\varkappa,\ell}%
\]
defined by $\mu((\alpha,h))=(\mathrm{Ker}(h),\mathrm{Ker}(h)^{\perp}%
,\alpha^{\perp})$. Let%
\[
\rho:V_{n-p+1,p+1+\varkappa,\ell}\rightarrow G_{n-p+1,\ell+p+1+\varkappa}%
\]
be the map defined by $\rho(a,b,c)=a$. Then $\rho\circ\mu$ gives a fiber
bundle. Since $\rho^{-1}(\mathbb{R}^{n-p+1}\times0)$ is $G_{p+1+\varkappa
,\ell}$ for $0\times\mathbb{R}^{\ell+p+1+\varkappa}$, we have that $(\rho
\circ\mu)^{-1}(\mathbb{R}^{n-p+1}\times0)$ is $V_{\ell+p+1+\varkappa
,p+1+\varkappa}$, namely $O(\ell+p+1+\varkappa)/O(\ell)$. Hence, $\rho\circ
\mu$ is a homotopy $\ell$-equivalence. We note that%
\begin{equation}
(\rho\circ\mu)^{\ast}(\widehat{\gamma}_{G_{n-p+1,\ell+p+1+\varkappa}}%
^{\ell+p+1+\varkappa})\approx\widehat{\gamma}_{V(\gamma_{G_{n+2+\varkappa}%
}^{n+2+\varkappa},\mathbb{R}^{p+1+\varkappa})}^{\ell}\oplus\varepsilon
_{V(\gamma_{G_{n+2+\varkappa}}^{n+2+\varkappa},\mathbb{R}^{p+1+\varkappa}%
)}^{p+1+\varkappa}.
\end{equation}

We have that $\pi_{n+\ell}\left(  T(\widehat{\gamma}_{V_{n+2}}^{\ell})\right)
$ is isomorphic to $\pi_{n+\ell+p+\varkappa}\left(  T(\widehat{\gamma
}_{V_{n+2}}^{\ell}\oplus\varepsilon_{V_{n+2}}^{p+\varkappa})\right)  $. This
is isomorphic to%
\[
\pi_{n+\ell+p+\varkappa}\left(  T((\widehat{\gamma}_{V_{n+2}}^{\ell}\oplus
(\pi_{P}^{V})^{\ast}(TP\oplus\nu_{P}))|_{V(\gamma_{G_{n+2+\varkappa}%
}^{n+2+\varkappa},TP\oplus\varepsilon_{P}^{1}\oplus\nu_{P})})\right)  \text{.}%
\]
Since setting $G=G_{n+2+\varkappa,\ell+p}$, we have the bundle map of%
\[
\widehat{\gamma}_{V_{n+2}}^{\ell}\oplus(\pi_{P}^{V})^{\ast}(TP\oplus\nu
_{P})\longrightarrow(\pi_{G}^{V})^{\ast}(\widehat{\gamma}_{G}^{\ell
+p})|_{V(\gamma_{G}^{n+2+\varkappa},\mathbb{R}^{p+1+\varkappa})}\times\nu_{P}%
\]
covering the canonical homotopy $(n+1+\varkappa)$-equivalent map%
\[
V(\gamma_{G_{n+2+\varkappa}}^{n+2+\varkappa},TP\oplus\varepsilon_{P}^{1}%
\oplus\nu_{P})\longrightarrow V(\gamma_{G}^{n+2+\varkappa},\mathbb{R}%
^{p+1+\varkappa})\times P,
\]
by (10.5), the last group is isomorphic to%
\[
\pi_{n+\ell+p+\varkappa}\left(  T((\pi_{G}^{V})^{\ast}\widehat{\gamma}%
_{G}^{\ell+p}|_{V(\gamma_{G}^{n+2+\varkappa},\mathbb{R}^{p+1+\varkappa}%
)})\wedge T(\nu_{P})\right)  \text{.}%
\]
This is isomorphic to%
\[
\pi_{n+\ell+p+1+2\varkappa}\left(  T((\pi_{G}^{V})^{\ast}(\widehat{\gamma}%
_{G}^{\ell+p}\oplus\varepsilon_{G}^{1+\varkappa})|_{V(\gamma_{G}%
^{n+2+\varkappa},\mathbb{R}^{p+1+\varkappa})})\wedge T(\nu_{P})\right)  .
\]
Since we have the bundle map of $\widehat{\gamma}_{G_{n+2+\varkappa}}^{\ell
}\oplus\varepsilon_{G_{n+2+\varkappa}}^{p+1+\varkappa}$ to $\widehat{\gamma
}_{G}^{\ell+p}\oplus\varepsilon_{G}^{1+\varkappa}$\ covering the canonical map
$G_{n+2+\varkappa}\rightarrow G$ again, this is isomorphic to%
\[
\pi_{n+\ell+p+1+2\varkappa}\left(  T((\pi_{G_{n+2+\varkappa}}^{V})^{\ast
}(\widehat{\gamma}_{G_{n+2+\varkappa}}^{\ell}\oplus\varepsilon
_{G_{n+2+\varkappa}}^{p+1+\varkappa})|_{V(\gamma_{G_{n+2+\varkappa}%
}^{n+2+\varkappa},\mathbb{R}^{p+1+\varkappa})})\wedge T(\nu_{P})\right)
\text{.}%
\]
Since $\rho\circ\mu$ is a homotopy $\ell$-equivalence, it follows from (10.6)
that the last group is isomorphic to%
\[
\pi_{n+\ell+p+1+2\varkappa}\left(  T(\widehat{\gamma}_{G_{n-p+1,\ell
+p+1+\varkappa}}^{\ell+p+1+\varkappa})\wedge T(\nu_{P})\right)  \text{,}%
\]
which is isomorphic to%
\begin{equation}
\pi_{n+\ell+\varkappa}\left(  T(\widehat{\gamma}_{G_{n-p+1}}^{\ell})\wedge
T(\nu_{P})\right)  \approx\{S^{n+\ell+\varkappa};T(\widehat{\gamma}%
_{G_{n-p+1}}^{\ell})\wedge T(\nu_{P})\}.
\end{equation}
Since $\ell$ is sufficiently large and $\varkappa$\ is any integer with
$\varkappa\geqq p+3$, we have proved the assertion.
\end{proof}

Now we define%
\[
B_{V}=\lim_{\ell\rightarrow\infty}C^{0}(S^{n-p+\ell},T(\widehat{\gamma
}_{\widetilde{G}_{n-p+1,\ell}}^{\ell})).
\]
By the duality map $T(\nu_{P})\wedge S^{t}(P^{0})\rightarrow S^{p+\varkappa
+t}$, the last group in (10.7) is isomorphic to%
\[
\{S^{n+\ell+\varkappa}\wedge S^{t}(P^{0});T(\widehat{\gamma}_{G_{n-p+1}}%
^{\ell})\wedge S^{p+\varkappa+t}\}\approx\{S^{n-p+\ell}\wedge P^{0}%
;T(\widehat{\gamma}_{G_{n-p+1}}^{\ell})\}.
\]
Then we have the following proposition.

\begin{proposition}
Let $P$ be oriented and connected. Then we have the isomorphism%
\[
\lim_{\ell\rightarrow\infty}\pi_{n+\ell}\left(  T(\widehat{\gamma}_{V_{n+2}%
}^{\ell})\right)  \longrightarrow\lbrack P,B_{V}],
\]
where $V_{n+2}=V(\gamma_{\widetilde{G}_{n+1,\ell}}^{n+1}\oplus\varepsilon
_{\widetilde{G}_{n+1,\ell}}^{1},TP\oplus\varepsilon_{P}^{1})$.
\end{proposition}

Furthermore, let $n=p$ and $F$ be the space defined in Introduction. Then it
has been proved in \cite[Proposition 4.1 and Remark 4.3]{FoldRIMS} and
\cite{FoldKyoto} by applying Theorem 10.4 that there exists an isomorphism%
\begin{equation}
\mathfrak{O}(n,P;\Omega^{1,0})\approx\lbrack P,F].
\end{equation}
In particular, the homotopy group $\pi_{n}(F,\ast)$ of the $F$'s connected
component of maps of degree $0$ is isomorphic to the $n$-th stable homotopy
group $\pi_{n}^{s}$ by \cite[Theorem 1 and Corollary 2]{FoldKyoto}.

Chess\cite[1.3 Corollary]{Che} has proved, in our notation, that
$\mathfrak{O}(n,\mathbb{R}^{n};\Omega^{1,0})$ is isomorphic to $\pi_{n}^{s}$.
We can prove this fact from (10.5). In fact, let $\mathfrak{O}(n,S^{n}%
;\Omega^{1,0};0)$\ denote the subset of $\mathfrak{O}(n,S^{n};\Omega^{1,0}%
)$\ which consists of all cobordism classes $[f]$\ such that the degree of
$f$\ is $0$. By applying the h-principles in (h-P) for fold-maps to
$\mathbb{R}^{n}$ and $S^{n}$ in \cite{Exist} we can prove that the inclusion
$\mathbb{R}^{n}=S^{n}\backslash\{(0,\cdots,0,1)\}\rightarrow S^{n}$
canonically induces an isomorphism%
\[
\mathfrak{O}(n,\mathbb{R}^{n};\Omega^{1,0})\rightarrow\mathfrak{O}%
(n,S^{n};\Omega^{1,0};0).
\]
The detail is left to the reader.

The author proposes a problem: Let $\mathfrak{i}_{\Omega}:\pi_{n}^{s}%
\approx\mathfrak{O}(n,S^{n};\Omega^{1,0};0)\rightarrow\mathfrak{O}%
(n,S^{n};\Omega)$ denote the homomorphism induced from the inclusion
$\Omega^{1,0}(n,n)\rightarrow\Omega(n,n)$. For an element $a\neq0$ in $\pi
_{n}^{s}$, we define a $\mathcal{K}$-invariant open subset $U(a)$ in
$J^{k}(n,n)$ as the union of all $\mathcal{K}$-invariant open sets
$\Omega(n,n)$ such that $\mathfrak{i}_{\Omega}(a)\neq0$. Study how the
singularities in $U(a)$ are related to $a$. The spaces $B_{\mathfrak{O}}$'s
for $\mathfrak{O}(n,S^{n};\Omega)$ will be useful. This should be compared
with the theme and the problem studied in \cite{StableHGS}.

\bigskip

\textbf{Acknowledgement}. The author would like to express his gratitude to
Professor O. Saeki for helpful discussions and comments.

\end{document}